\documentclass[10pt, twosides]{amsart}

\usepackage{defs}

\title[PMP under frequency constraints]{Discrete time Pontryagin maximum principle for optimal control problems under state-action-frequency constraints}
\author[P. Paruchuri]{Pradyumna Paruchuri}
\author[D. Chatterjee]{Debasish Chatterjee}
\thanks{The authors are with Systems \& Control Engineering, IIT Bombay, Powai, Mumbai 400076, India. They were supported in part by the grant 17ISROC001 from the Indian Space Research Organization. The authors thank Masaaki Nagahara for helpful discussions during the early phases of this work, and Navin Khaneja for suggesting the technique behind the second proof of our main result.}
\thanks{Emails: \textsf{pradyumn@sc.iitb.ac.in}, \textsf{dchatter@iitb.ac.in}}

\begin{document}

	\begin{abstract}
		We establish a Pontryagin maximum principle for discrete time optimal control problems under the following three types of constraints: a) constraints on the states pointwise in time, b) constraints on the control actions pointwise in time, and c) constraints on the frequency spectrum of the optimal control trajectories. While the first two types of constraints are already included in the existing versions of the Pontryagin maximum principle, it turns out that the third type of constraints cannot be recast in any of the standard forms of the existing results for the original control system. We provide two different proofs of our Pontryagin maximum principle in this article, and include several special cases fine-tuned to control-affine nonlinear and linear system models. In particular, for minimization of quadratic cost functions and linear time invariant control systems, we provide tight conditions under which the optimal controls under frequency constraints are either normal or abnormal.
	\end{abstract}

	\maketitle

	\section{Introduction}
		\label{sec:the intro}

As control engineers we encounter various types of constraints in control systems for a plethora of reasons: limitations on the magnitude of actuator outputs are almost omnipresent; bounds on the state variables of, e.g., robotic arms and chemical plants, should be ensured for safety considerations; satellites that image particular geographical areas of the earth must orient themselves and point at precise coordinates at pre-specified instants of time, etc. While constrained control problems are difficult in general, and this is evidenced by the fact that the literature on unconstrained control problems by far outweighs that on constrained problems, control synthesis techniques that account for all possible constraints are bootstrapped to result in greater accuracy due to increased awareness of the actuator limitations and foresight. The burgeoning demand for execution of precise control tasks necessitates the development of tools that permit the inclusion of such constraints at the synthesis stage, and in this respect, inclusion of control frequency constraints is a natural direction to pursue.

Optimal control theory provides us with a set of sophisticated and powerful tools to design controllers under an array of constraints, and also to boost performance by taking account of such constraints on the states and the control actions in time domain. These techniques typically rely on the assumption that the values attained by the candidate control functions can be changed arbitrarily quickly over time, but such an assumption rarely holds true in practice. In particular, inertial actuators such as robotic arms, rotating machines, etc., cannot faithfully execute control commands that demand very quick transitions between different control values. Such issues naturally lead to lacunae between the control commands received at the actuators and those that are faithfully executed, thereby contributing to loss of precision and the emergence of differences between desired and observed outputs.

This article addresses a class of optimal control problems that includes constraints on the frequency of admissible control functions in addition to state and control constraints. More specifically, we address optimal control problems for discrete-time nonlinear smooth control systems with the following three important classes of constraints:
\begin{enumerate}[label=(\Roman*), leftmargin=*, widest=III, align=left]
	\item \label{constr:states} constraints on the states at each time instant,
	\item \label{constr:controlmag} constraints on the control magnitudes at each time instant, and
	\item \label{constr:controlfreq} constraints on the frequency of the control functions.
\end{enumerate}
Constraints on the states (as in \ref{constr:states}) are desirable and/or necessary in most applications; the class of constraints treated here are capable of describing a general class of path-planning objectives, and subsumes both ballistic and servomechanism reachability problems. Constraints on the control magnitudes (as in \ref{constr:controlmag}) are typically simpler to deal with compared to state constraints; in particular, the two general techniques for synthesis of optimal controls, namely, dynamic programming and the maximum principle \cite{ref:Lib-12}, are capable of dealing with these constraints with relative ease.

Constraints on the control frequencies (as in \ref{constr:controlfreq}), in contrast to the other two types of constraints, are rarely encountered in the theory despite the fact that control theory started off with the so-called \emph{frequency-domain} techniques. A well-known and widely employed control strategy that treats frequency-domain properties of control functions is the so-called $H^\infty$ control \cite{ZDG}, 
but these techniques can neither treat pre-specified hard bounds on the frequency components in the control signals, nor are they capable of admitting state or control constraints at the synthesis stage. Frequency constraints can be indirectly addressed in $H^\infty$ control through penalization of appropriate $H^\infty$ norms, but such designs rely on heuristics and many trial-and-error steps. To the best of our knowledge, except for a US patent \cite{ScaBro95} where frequency constraints were imposed specifically to avoid a resonance mode in the arm of the read head positioner in a disk drive, there has been no systematic investigation into control with mixed frequency and time-domain constraints.

The celebrated Pontryagin maximum principle \cite{Boltyanskii}, a central tool in optimal control theory, provides first order necessary conditions for optimal controls. These necessary conditions, or equivalently, characterizations of optimal controls, serve to narrow the search space over which algorithms can play and extract optimal controls. The discrete time Pontryagin maximum principle was developed primarily by Boltyanskii (see \cite{ref:Bol-75, Boltyanskii} and the references therein), with several early refinements reported in \cite{ref:DubMil-65, ref:DzjPsh-75, ref:Dub-78}, and perhaps the most recent extensions appearing in \cite{ref:BouTre-16}; see \cite{ref:Psh-71} for a careful discussion about the differences between continuous and discrete time versions of the Pontryagin maximum principle. While these versions of the Pontryagin maximum principle are capable of handling constraints of the form \ref{constr:states} and \ref{constr:controlmag}, the new ingredient in this article is the set of frequency constraints \ref{constr:controlfreq}. We formulate frequency constraints on the control functions in terms of the \emph{active support set} --- the set on which the Fourier transform of the control function is allowed to take non-zero values. We engineer band-limited controls via appropriately defining the active sets; the constraints may be selected based on specific features or physics of the actuators, thereby ensuring faithful execution of the control commands. Our main result --- Theorem \ref{th:main pmp} in \secref{sec:main result} --- is a Pontryagin maximum principle for discrete-time nonlinear control systems with smooth data under all the three types of constraints \ref{constr:states}, \ref{constr:controlmag}, and \ref{constr:controlfreq}.

This maximum principle yields a well-defined two-point boundary value problem, which may serve as a starting point for algorithms such as shooting techniques that typically employ variants of Newton methods, to arrive at optimal control functions. If a solution of the two-point boundary value problem is found, feasibility of the original optimal control problem is automatically established. However, since the maximum principle provides (local) necessary conditions for optimality, not all solutions may achieve the minimum cost, and further analysis may be needed to select the cost-minimizing controls. A number of special cases of the main result, dealing with control-affine nonlinear systems, time-varying linear systems, etc., are provided in \secref{sec:corollaries}, and the important special case of optimal control of linear time-invariant control systems under quadratic costs and frequency constraints is treated in \secref{sec:LQR problems}. Two different proofs of Theorem \ref{th:main pmp} are provided in Appendix \ref{sec:proofs}, and Appendix \ref{sec:corollary proofs} contains the proofs of the various special cases. The necessary prerequisites for the proofs are reviewed in Appendices \ref{sec:cones}-\ref{sec:tents}.

\subsection*{Notation}
We employ standard notation: \(\Nz\) denotes the non-negative integers, \(\N\) the positive integers, \(\R\) the real numbers, and \(\C\) the complex numbers. We denote by \(\preceq\) the standard partial order on the set \(\R^{n}\) induced by the non-negative orthant: for \(a,\; b \in \R^{n}\), \(a \preceq b\) iff \(a_{i} \le b_{i}\) for every \(i = 1, \ldots, n\); we sometimes write \(b \succeq a\) to express the same statement. For us \(\ii \Let \sqrt{-1}\) is the unit complex number, \(I_n\) is the \(n\times n\) identity matrix. The vector space \(\R^{\genDim}\) is always assumed to be equipped with the standard inner product \(\inprod{v}{v'} \Let v \transp v'\) for every \(v, v' \in \R^{\genDim}\). In the theorem statements we use \(\dualSpace{\bigl(\R^{\genDim} \bigr)}\) to denote the dual space of \(\R^{\genDim}\) for the sake of precision; of course, \(\dualSpace{\bigl(\R^{\genDim}\bigr)}\) is isomorphic to \(\R^{\genDim}\) in view of the Riesz representation theorem. 

	\section{Problem Setup}
		\label{sec:opt prob}
		Consider a discrete time control system described by
\begin{equation}
\label{e:gen sys}
	\state_{t+1} = \sysDyn (\state_{t}, \conInp_{t}) \quad \text{for } t = 0, \ldots, \horizon-1,
\end{equation}
where \(\state_{t} \in \R^{\sysDim}\) and \(\conInp_{t} \in \R^{\conDim}\) and \( (\sysDyn)_{t=0}^{\horizon-1}\) is a family of maps such that \(\R^{\sysDim} \times \R^{\conDim} \ni (\xi, \mu) \mapsto \sysDyn[s](\xi, \mu) \in \R^{\sysDim}\) is continuously differentiable for each \(s = 0, \ldots, \horizon-1\).

Let \(\conInp \kth \Let (\conInp_{t} \kth)_{t = 0}^ {\horizon-1}\) denote the \(k^\text{th}\) control sequence, and \(\freqComp\) denote its discrete Fourier transform (DFT). The relationship between \(\freqComp\) and \(\conInp \kth\) is given by \cite[Chapter 7]{Stein-Shakarchi}:
\begin{equation}
\label{e:frequency components}
\begin{aligned}
	\freqComp \Let (\freqComp_{\xi})_{\xi=0}^{\horizon-1} = \biggl( \sum_{t = 0}^{\horizon-1} \conInp_{t}\kth \epower{-\ii 2\pi \xi t/\horizon} \biggr)_{\xi=0}^{\horizon-1} \quad & \text{for } \xi = 0, \ldots, \horizon-1 \\
	& \text{and } k = 1, \ldots, \conDim.
\end{aligned}
\end{equation}

In the context of \eqref{e:gen sys}, the objective of this article is to characterize solutions of the finite horizon constrained optimal control problem:
\begin{equation}
\label{e:abstract prob}
\begin{aligned}
	& \minimize_{(\conInp_{t})_{t=0}^{\horizon -1}} && \sum_{t=0}^{\horizon-1} \cost (\state_{t}, \conInp_{t})\\
	& \sbjto &&
	\begin{cases}
		\text{dynamics \eqref{e:gen sys}},\\
		\text{state constraints at each stage } t = 0, \ldots, \horizon,\\
		\text{control constraints at each stage } t = 0, \ldots, \horizon-1,\\
		\text{constraints on frequency components of the control sequence.}
	\end{cases}
\end{aligned}
\end{equation}
where \(\horizon \in \N\) is fixed, and \(\R^{\sysDim} \times \R^{\conDim} \ni (\xi, \mu)\mapsto \cost(\xi, \mu) \in \R\) is a continuously differentiable function representing the stage cost at time \(t\), and \(t = 0, \ldots, \horizon-1\).

The three classes of constraints considered in the optimal control problem \eqref{e:abstract prob} are as follows:
\begin{enumerate}[label=(\roman*), leftmargin=*, align=left, widest=iii]
	\item \emph{Control constraints}: \(\admControl_{t} \subset \R^{\conDim}\) is a given non-empty set for each \(t = 0, \ldots, \horizon\). We impose the constraints that the control action \(\conInp_{t}\) at stage \(t\) must lie in \(\admControl_{t}\):
	\begin{equation}
	\label{e:control constr.}
	\begin{aligned}
		\conInp_{t} \in \admControl_{t} \quad \text{for } t = 0, \ldots, \horizon-1.
	\end{aligned}
	\end{equation}

\item \emph{State constraints}: Let \(\admStates_{t} \subset \R^{\sysDim}\) be a given non-empty set for each \(t = 0, \ldots, \horizon\). We shall restrict the trajectory of the states \((\state_{t})_{t=0}^{\horizon}\) to the tube \( \admStates_{0} \times \admStates_{1} \times \cdots \times \admStates_{\horizon} \subset (\R^{\sysDim})^{\horizon+1}\); 
	\begin{equation}
	\label{e:state constr.}
		\state_{t} \in \admStates_{t} \quad \text{for } t = 0, \ldots, \horizon.
	\end{equation}

	\item \emph{Frequency constraints}: For a control sequence \(\conInp \kth\) we define \(\admFreq \kth \subset \C^{\horizon}\) to be the set of permissible frequency components \(\freqComp = (\freqComp_{\xi})_{\xi=0}^{\horizon-1}\). The set \(\admFreq \kth\) is constructed such that it allows non-zero components only in the selected frequencies. For a vector \(v \in \C^{\horizon}\) we define its support as
	\[
		\support(v) \Let \set[\big]{ i \in \{1, \ldots, \horizon \} \suchthat v_{i} \not = 0}.
	\]
	We stipulate that
	\begin{equation}
	\label{e:freq constr.}
		\freqComp \in \admFreq \kth \Let \set[\big]{ v \in \C^{\horizon} \suchthat \support(v) \subset W \kth},
	\end{equation}
	where \(W \kth \subset \{1, \ldots, \horizon \}\) represents the support for the selected frequencies in the \(k^\text{th}\) control sequence. The sets \(\bigl(W \kth \bigr)_{k=1}^{\conDim}\) are assumed to be given as part of the problem specification.
\end{enumerate}

The standard DFT relation in \eqref{e:frequency components} can be written in a compact form as:
\begin{equation}
\label{e:frequency matrix}
	\freqComp = \dft \conInp \kth \quad \text{for } k = 1, \ldots, \conDim,
\end{equation}
where
\[ \conInp \kth \Let \columnVector{\conInp \kth}{\horizon} \in \R^{\horizon}, \quad \freqComp \Let \columnVector{\freqComp}{\horizon} \in \C^{\horizon}, \text{ and} \]
\[ \dft \Let \frac{1}{\sqrt{\horizon}} \DFT{\omega}{\horizon} \in \C^{\horizon \times \horizon}, \]
and \(\omega \Let \epower{-\ii 2\pi /\horizon}\) is a primitive \(\horizon\)-th root of unity. In order to visualize the frequency components in all the control inputs, we represent the combined control profile in the following (stacked) fashion:
\begin{equation}
\label{e:combined profile}
	\fullConInp \Let \pmat{\conInp \kth[1] \\ \vdots \\ \conInp \kth[\conDim]} \in \R^{\conDim \horizon} \quad \text{and} \quad \fullFreqComp \Let \pmat{\freqComp[1] \\ \vdots \\ \freqComp[\conDim]} \in \C^{\conDim \horizon}.
\end{equation}
In terms of the representations \eqref{e:combined profile}, the relation \eqref{e:frequency matrix} can be written in a compact way as:
\begin{equation}
\label{e:full frequency profile}
	\fullFreqComp = \diagDFT \fullConInp, \quad \text{with } \diagDFT \Let \blkdiag (\dft, \ldots, \dft) \in \C^{\conDim \horizon \times \conDim \horizon}.
\end{equation}

Since \(\fullConInp\) is a vector with real entries, the real and imaginary parts of the frequency components can be separated by considering the real and imaginary parts in the matrix \(\diagDFT\) individually. To impose the given frequency constraints and yet work with real numbers only, we separate out the real and imaginary parts. We define a band-stop filter \(\bandStop \Let \stopFilter \fullFreqComp\), where \( \stopFilter \Let \blkdiag (\stopFilter \kth[1], \ldots, \stopFilter \kth[\conDim] ), \) with each of the \(\stopFilter \kth\) formed by the rows \(\mathbf{e}_{\xi}\) of \(I_{\horizon}\) for \(\xi \not \in W \kth \). The constraints \eqref{e:freq constr.} on the frequency components of the control now translate to:
\begin{equation}
\label{e:allowed frequencies}
	\stopFilter \fullFreqComp = 0 \quad 
	 \Leftrightarrow \quad \pmat{\stopFilter \diagDFT_{\text{real}} \\ \stopFilter \diagDFT_{\text{imag}}} \fullConInp = 0.
\end{equation}

Define \(\mathscr{F} \Let \pmat{\stopFilter \diagDFT_{\text{real}} \\ \stopFilter \diagDFT_{\text{imag}}}\) and let \(\linTran \in \R^{\conDim \horizon \times \conDim \horizon}\) denote the matrix that maps the vector \(\fullConInp\) to \(\pmat{\conInp_{0} & \ldots & \conInp_{\horizon-1}} \transp\):
\begin{equation}
\label{e:permutation}
	\linTran \pmat{\conInp \kth[1] \\ \vdots \\ \conInp \kth[\conDim]} = \pmat{\conInp_{0} \\ \vdots \\ \conInp_{\horizon-1}}
\end{equation}
Observe that \(\linTran\) is non-singular since the transformation representing \(\linTran\) is a permutation matrix, and in particular is a bijection. Then we can write the frequency constraints in \eqref{e:freq constr.} as
\[
	\mathscr{F} \linTran \inverse \pmat{\conInp_{0} \\ \vdots \\ \conInp_{\horizon-1}} = 0.
\]
Eliminating, if necessary, the zero rows of the matrix \(\mathscr{F}\), our constraint takes the form
\begin{equation}
\label{e:full freq constraints}
	\sum_{t=0}^{\horizon-1} \freqDer \conInp_{t} = 0.
\end{equation}
where \(\freqDer \in \R^{\freqDim \times \conDim}\) represents the corresponding columns of \(\mathscr{F} \linTran \inverse\) that multiply \(\conInp_{t}\). In other words, there exists a linear map \(\freqConstr : \conDim \horizon \lra \conDim \horizon\) that  describes the  constraints on the frequency spectrum of the control trajectory \((\conInp_{t})_{t=0}^{\horizon-1}\) as the following equality constraint:
\begin{equation}
\label{e:freq assumptions}
\freqConstr (\conInp_{0}, \ldots, \conInp_{\horizon-1}) = \sum_{t=0}^{\horizon-1} \freqDer \conInp_{t} = 0 \quad \text{for \(\bigl(\freqDer[t]\bigr)_{t=0}^{\horizon-1} \subset \R^{\freqDim \times \conDim \horizon}\) as in \eqref{e:full freq constraints}}.
\end{equation}
We shall refer to \(\freqConstr\) as our \embf{frequency constraint map}.

The abstract optimal control problem \eqref{e:abstract prob} can now be formally written as:
\begin{equation}
\label{e:opt prob}
\begin{aligned}
& \minimize_{(\conInp_{t})_{t=0}^{\horizon -1}} && \sum_{t=0}^{\horizon-1} \cost (\state_{t}, \conInp_{t})\\
	& \sbjto &&
	\begin{cases}
		\text{dynamics \eqref{e:gen sys}},\\
		\state_{t} \in \admStates_{t} \quad \text{for } t = 0, \ldots, \horizon,\\
		\conInp_{t} \in \admControl_{t} \quad \text{for } t = 0, \ldots, \horizon-1,\\
		\freqConstr  (\conInp_{0}, \ldots, \conInp_{\horizon-1}) = 0,
	\end{cases}
\end{aligned}
\end{equation}
with the following data:
\begin{enumerate}[label=(\ref{e:opt prob}-\alph*), leftmargin=*, align=left]
	\item \(\horizon \in \N\) is fixed;
	\item \(\R^{\sysDim} \times \R^{\conDim} \ni (\xi, \mu) \mapsto \cost(\xi, \mu) \in \R\) is a continuosly differentiable function for each \(t = 0, \ldots, \horizon-1\);
	\item \(\admStates_{t}\) is a subset of \(\R^{\sysDim}\) for each \(t\);
	\item \(\admControl_{t}\) is a subset of \(\R^{\conDim}\) for each \(t\);
	\item \(\freqConstr : \R^{\conDim \horizon} \lra \R^{\freqDim}\) is a given linear map on the control trajectory \(\conInp_{0}, \ldots, \conInp_{\horizon-1}\) for some \(\freqDim \in \N\).
\end{enumerate}

An optimal solution \((\optCon)_{t=0}^{\horizon-1}\) of \eqref{e:opt prob} is a sequence in \(\prod_{i=0}^{\horizon-1} \admControl_{i}\), and it generates its corresponding optimal state trajectory \((\optState)_{t=0}^{\horizon}\) according to \eqref{e:gen sys}. The pair \( \bigl( (\optState)_{t=0}^\horizon, (\optCon)_{t=0}^{\horizon-1} \bigr)\) is called an \embf{optimal state-action trajectory}.

\begin{remark}
	Constraints on the control frequencies cannot in general be translated into equivalent constraints on the control actions and/or the states of the original system. Had that been possible, the standard PMP would have sufficed. To see this negative assertion, consider the simple case that the system \eqref{e:gen sys} is linear and time-invariant, i.e., \(\sysDyn(\xi, \mu) = A\xi + B \mu\) for all \(t\) and for some fixed \(A \in \R^{\sysDim \times \sysDim}\) and \(B \in \R^{\sysDim \times \conDim}\). Assume further that the frequency constraint map \(\freqConstr\) is a bijection. Even then the constraint on the control actions \(\freqConstr (\conInp_{0}, \ldots, \conInp_{\horizon-1}) = 0\) cannot in general be transformed into equivalent constraints on the states of the form \((\state_{t})_{t=1}^{\horizon} \in \admStates \subset \R^{\sysDim \horizon}\). Indeed, when \(\horizon > \sysDim\), constraints on the control actions can only be contained in constraints of the form \((\state_{t})_{t=1}^{\horizon} \in \admStates \subset \R^{\sysDim \horizon}\) since the transformation from the control trajectory \((\conInp_{t})_{t=0}^{\horizon-1}\) to state trajectory \((\state_{t})_{t=1}^{\horizon}\) is not a bijection. A fresh investigation is, therefore, needed. The standard PMP \cite[Theorem 20]{ref:Bol-75} deals with constraints on the states and control actions that are expressed pointwise in time. Since constraints on the frequency components of the control, by definition, bring in dependence among the control actions at each time, the standard Hamiltonian maximization condition \cite[Theorem 20 (C)]{ref:Bol-75} cannot be used as is.
\end{remark}

	\section{Main Result}
		\label{sec:main result}

	The following theorem provides first order necessary conditions for optimal solutions of \eqref{e:opt prob}; it is the main result of this article.
	\begin{theorem}[Pontryagin maximum principle under state-action-frequency constraints]
	\label{th:main pmp}
		Let \(\bigl((\optState)_{t=0}^{\horizon}, (\optCon)_{t=0}^{\horizon-1} \bigr)\) be an optimal state-action trajectory for \eqref{e:opt prob} with \(\freqConstr\) as defined in \eqref{e:freq assumptions}. Define the Hamiltonian
		\begin{equation}
		\label{e:hamiltonian}
		\begin{aligned}
			& \R \times \dualSpace{\bigl(\R^{\freqDim} \bigr)} \times \dualSpace{\bigl(\R^{\sysDim}\bigr)} \times \Nz \times \R^{\sysDim} \times \R^{\conDim} \ni  (\genCost, \genFreq, \genDyn, \genTime, \genState, \genCon ) \mapsto\\
			& \qquad \hamiltonian[\genCost, \genFreq] ( \genDyn, \genTime, \genState, \genCon  ) \Let \inprod{\genDyn}{\sysDyn[\genTime](\genState, \genCon)} - \genCost \cost[\genTime](\genState, \genCon) - \inprod{\genFreq}{\freqDer[s] \genCon} \in \R.
		\end{aligned}
		\end{equation}
		Then 
		there exist
		\begin{itemize}[label=\(\circ\), leftmargin=*]
			\item a trajectory \(\bigl(\adjDyn \bigr)_{t=0}^{\horizon-1} \subset \dualSpace{\bigl(\R^{\sysDim} \bigr)}\),
			\item a sequence \(\bigl(\adjState \bigr)_{t=0}^{\horizon} \subset \dualSpace{\bigl(\R^{\sysDim}\bigr)}\), and
			\item a pair \(\bigl(\adjCost, \adjFreq \bigr) \in \R \times \dualSpace{\big(\R^{\freqDim}\bigr)}\),
		\end{itemize}
		satisfying the following conditions:
		\begin{enumerate}[label={\textup{(PMP-\roman*)}}, leftmargin=*, align=left, widest=iii]
			\item \label{pmp:non-negativity} non-negativity condition
				\begin{quote}
					\(\adjCost \ge 0;\)
				\end{quote}
			\item \label{pmp:non-triviality} non-triviality condition
				\begin{quote}
					the adjoint trajectory \(\bigl(\adjDyn\bigr)_{t=0}^{\horizon-1}\) and the pair \(\bigl(\adjCost, \adjFreq\bigr)\) do not simultaneously vanish;
				\end{quote}
			\item \label{pmp:optimal dynamics} state and adjoint system dynamics
				\begin{alignat*}{2}
					\optState[t+1] & = \derivative{\genDyn}{\hamiltonian[\adjCost, \adjFreq] \bigl(\adjDyn, t, \optState, \optCon \bigr)} \quad && \text{for } t = 0, \ldots, \horizon-1,\\
					\adjDyn[t-1] & = \derivative{\genState}{\hamiltonian[\adjCost, \adjFreq] \bigl(\adjDyn, t, \optState, \optCon \bigr)} - \adjState \quad && \text{for } t = 1, \ldots, \horizon-1,
				\end{alignat*}
				where \(\adjState\) lies in the dual cone of a tent \(\stateTent( \optState )\) of \(\admStates_{t}\) at \(\optState\);
			\item \label{pmp:transversality} transversality conditions
				\begin{align*}
					\derivative{\genState}{\hamiltonian[\adjCost, \adjFreq] \bigl(\adjDyn[0], 0, \optState[0], \optCon[0] \bigr)} - \adjState[0] = 0 \qquad \text{and} \qquad \adjDyn[\horizon-1] = -\adjState[\horizon],
				\end{align*}
				where \(\adjState[0]\) lies in the dual cone of a tent \(\stateTent[0] (\optState[0])\) of \(\admStates_{0}\) at \(\optState[0]\) and \(\adjState[\horizon]\) lies in the dual cone of a tent \(\stateTent[\horizon] (\optState[\horizon])\) of \(\admStates_{\horizon}\) at \(\optState[\horizon]\);
			\item \label{pmp:Hamiltonian VI} Hamiltonian maximization condition, pointwise in time,
				\begin{align*}
					\inprod{\derivative{\genCon}{\hamiltonian[\adjCost, \adjFreq] \bigl(\adjDyn, t, \optState, \optCon \bigr)}}{\perturb{\conInp}_{t}} \le 0 \quad \text{whenever } \optCon + \perturb{\conInp}_{t} \in \conTent ( \optCon ),
				\end{align*}
				where \(\conTent(\optCon)\) is a local tent at \(\optCon\) of the set \(\admControl_{t}\) of admissible actions;
			\item \label{pmp:freq} frequency constraints
				\begin{align*}
					\freqConstr \bigl(\optCon[0], \ldots, \optCon[\horizon-1]\bigr) = 0.
				\end{align*}
		\end{enumerate}
	\end{theorem}

	We present a complete proof of Theorem \ref{th:main pmp} in \appref{sec:proofs}. The rest of this section is devoted to a scrutiny of various facets of Theorem \ref{th:main pmp} over a sequence of remarks, and providing a set of corollaries catering to various special cases.

	\begin{remark}
		\label{rm:extremals}
		It is readily observed that since the scalar \(\adjCost\) and the vectors \(\adjDyn, \adjState, \adjFreq\) enter linearly in the Hamiltonian function \(\hamiltonian\), the non-negativity condition \ref{pmp:non-negativity} on \(\adjCost\) can be equivalently posed as the condition that \(\adjCost \in \set{0, 1}\). A quintuple \(\Bigl(\adjCost, \adjFreq, \bigl(\adjDyn\bigr)_{t=0}^{\horizon-1}, (\optState)_{t=0}^{\horizon}, (\optCon)_{t=0}^{\horizon-1}\Bigr)\) that satisfies the PMP is called an \embf{extremal lift} of the optimal state-action trajectory \(\bigl((\optState)_{t=0}^{\horizon}, (\optCon)_{t=0}^{\horizon-1}\bigr)\). Extremal lifts with \(\adjCost = 1\) are called \embf{normal} extremals and the ones with \(\adjCost = 0\) are caled \embf{abnormal} extremals
	\end{remark}

	\begin{remark}
		The term \(\inprod{\adjFreq}{\freqDer \genCon}\) in the Hamiltonian \(\hamiltonian\) is an additional term compared to the usual Hamiltonian formulation and corresponds to the constraints on the frequency components of the control sequence. Observe that since this term does not enter the conditions \ref{pmp:non-negativity}, \ref{pmp:optimal dynamics} and \ref{pmp:transversality}, the state and adjoint dynamics are unaffected. The element \(\adjFreq\) is a new entity in Theorem \ref{th:main pmp} compared toth e classical PMP in \cite{Boltyanskii}.
	\end{remark}

	\begin{remark}
		From the definition of the Hamiltonian function \(\hamiltonian\) we see that
		\[
			\derivative{\genDyn}{\hamiltonian[\adjCost, \adjFreq] \bigl(\adjDyn, t, \optState, \optCon\bigr)} = \sysDyn (\optState, \optCon).
		\]
		In other words, the state dynamics prescribed in \ref{pmp:optimal dynamics} simply states that the optimal state-action trajectory \(\bigl((\optState)_{t=0}^{\horizon}, (\optCon)_{t=0}^{\horizon-1}\bigr)\) satisfies the system dynamics \eqref{e:gen sys}.
	\end{remark}

	\begin{remark}
		The tents \(\stateTent(\optState)\) and \(\conTent(\optCon)\) mentioned in \ref{pmp:optimal dynamics} and \ref{pmp:Hamiltonian VI} are linear approximations of the sets \(\admStates_{t}\) and \(\admControl_{t}\) locally at \(\optState\) and \(\optCon\) respectively. Precise definitions of these tents will be given in \appref{sec:tents}. Intuitively, a tent (to a set at a point) consists of a set of directions along which it is possible to enter the set from that point. By construction a tent to a set at a point is a convex cone. The dual cone of a cone \(C\) is the convex cone that consists of all the directions along which one can most efficiently exit/leave the cone \(C\). The vectors \(\adjState\) lying in dual cones of a tent \(\stateTent(\optState)\) of \(\admStates_{t}\) at \(\optState\) represent the directions along which one can leave the set \(\admStates_{t}\) most efficiently from \(\optState\). A detailed exposition of dual cones and tents is given in \appref{sec:cones} and \appref{sec:tents} respectively.
	\end{remark}

	\begin{remark}
		In simple terms, the condition \ref{pmp:Hamiltonian VI} means that along the directions entering the set \(\admControl_{t}\) from \(\optCon\), the Hamiltonian \(\hamiltonian\) does not increase locally. We have used the name "Hamiltonian maximization condition" for this condition; although not entirely apt, it is borrowed from the continuous time counterpart of the Pontryagin maximum principle where the optimal control at time \(t\) maximizes the Hamiltonian at that instant \(t\) over the admissible action set. At the level of generality of Theorem \ref{th:main pmp}, an actual Hamiltonian maximization does not hold. However, such a maximization condition does indeed materialize under additional structural assumptions on the sets of admissible actions, as described in Corollary \ref{co:con-affine pmp}.
	\end{remark}

	\begin{remark}
		The conditions \ref{pmp:non-negativity} - \ref{pmp:freq} together constitute a well-defined two point boundary value problem with \ref{pmp:transversality} giving the entire set of boundary conditions. Newton-lie methods may be employed to solve this (algebraic) two point boundary value problem; see, eg., \cite[\S 2.4]{Trelat} for an illuminating discussion in the context of continuous-time problems. Solution techniques for two point boundary value problems is an active active field of research.
	\end{remark}

	\begin{remark}
	\label{rm:uncertainty}
		Uncertainty principles in time-frequency analysis impose fundamental restrictions on the classes of control magnitude and frequency constraints. For instance, the Donoho-Stark uncertainty principle \cite{ref:DonSta-89} shows that every non-zero \(\C\)-valued function \(g : \set[\big]{0, 1, \ldots, N-1} \mapsto \C\) must satisfy \(\support(g) + \support(\hat{g}) \ge 2 \sqrt{\abs{N}}\).\footnote{Further refinements due to Biro-Meshulam-Tao may be found in \cite{ref:Tao-05}; see also \cite{ref:MurWha-12} for a recent generalization.} Applied to the control trajectories \(\bigl(\conInp \kth\bigr)_{k=1}^{\conDim}\), one immediately finds that imposing certain types of control magnitude and frequency constraints simultaneously may lead to empty feasible sets of controls irrespective of the dynamics and other constraints. In other words, sufficient care needs  to be excercised to ensure a well-posed optimal control problem.
	\end{remark}

		\label{sec:corollaries}
		We now describe a few special cases of Theorem \ref{th:main pmp} that are fine-tuned to specific classes of control systems.

Consider a discrete-time control-affine system described by:
\begin{equation}
\label{e:con-affine}
	\state_{t+1} = \affSysDyn(\state_{t}) + \affSysCon(\state_{t}) \ \conInp_{t} \quad \text{for } t = 0, \ldots, \horizon-1,
\end{equation}
where \(\state_{t} \in \R^{\sysDim}\) and \(\conInp_{t} \in \R^{\conDim}\), and \((\affSysDyn)_{t=0}^{\horizon-1}\) and \((\affSysCon)_{t=0}^{\horizon-1}\) are two families of maps such that \(\R^{\sysDim} \ni \genState \mapsto\affSysDyn[s] (\genState) \in \R^{\sysDim}\) and \(\R^{\sysDim} \ni \genState \mapsto\affSysCon[s] (\genState) \in \R^{\sysDim \times \conDim}\) are continuously differentiable for each \(s = 0, \ldots, \horizon-1\). Consider the optimal control problem \eqref{e:opt prob} with the dynamics given by \eqref{e:con-affine}:
\begin{equation}
\label{e:aff opt prob}
\begin{aligned}
	& \minimize_{ (\conInp_{t} )_{t=0}^{\horizon -1}} && \sum_{t=0}^{\horizon-1} \affCost (\state_{t}, \conInp_{t} )\\
	& \sbjto &&
	\begin{cases}
		\text{dynamics \eqref{e:con-affine}},\\
		\state_{t} \in \admStates_{t} \quad \text{for } t = 0, \ldots, \horizon,\\
		\conInp_{t} \in \admControl_{t} \quad \text{for } t = 0, \ldots, \horizon-1,\\
		\freqConstr  (\conInp_{0}, \ldots, \conInp_{\horizon-1} ) = 0,\\
		\affCost (\genState, \cdot) : \admControl_{t} \ra \R \text{ is convex whenever }\genState \in \R^{\sysDim}, t=0, \ldots, \horizon-1,\\
		\admControl_{t} \text{ convex, compact, and non-empty for each }t = 0, \ldots, \horizon-1.
	\end{cases}
\end{aligned}
\end{equation}

\begin{corollary}[PMP for control-affine systems]
\label{co:con-affine pmp}
	Let \(\bigl((\optState)_{t=0}^{\horizon}, (\optCon)_{t=0}^{\horizon-1} \bigr)\) be an optimal state-action trajectory for \eqref{e:aff opt prob} with \(\freqConstr\) as defined in \eqref{e:freq assumptions}. Define the Hamiltonian
	\begin{equation}
	\label{e:aff ham}
	\begin{aligned}
		& \R \times \dualSpace{\bigl(\R^{\freqDim}\bigr)} \times \dualSpace{\bigl(\R^{\sysDim}\bigr)} \times \Nz \times \R^{\sysDim} \times \R^{\conDim} \ni  (\genCost, \genFreq, \genDyn, \genTime, \genState, \genCon ) \mapsto\\
		& \qquad \hamiltonian[\genCost, \genFreq] ( \genDyn, \genTime, \genState, \genCon  ) \Let \inprod{\genDyn}{\affSysDyn[\genTime](\genState) + \affSysCon[\genTime](\genState) \ \genCon} - \genCost \cost[\genTime](\genState, \genCon) - \inprod{\genFreq}{\freqDer[s] \genCon} \in \R.
	\end{aligned}
	\end{equation}
	Then there exist
	\begin{itemize}[label=\(\circ\), leftmargin=*]
		\item a trajectory \( \bigl(\adjDyn \bigr)_{t=0}^{\horizon-1} \subset \dualSpace{\bigl(\R^{\sysDim}\bigr)}\),
		\item a sequence \(\bigl(\adjState\bigr)_{t=0}^{\horizon} \subset \dualSpace{\bigl(\R^{\sysDim}\bigr)}\), and
		\item a pair \(\bigl(\adjCost, \adjFreq \bigr) \in \R \times \dualSpace{\bigl(\R^{\freqDim}\bigr)}\),
	\end{itemize}
	satisfying the following conditions:
	\begin{enumerate}[label={\textup{(AFF-\roman*)}}, leftmargin=*, align=left, widest=iii]
		\item \label{aff:non-negativity} non-negativity condition
			\begin{quote}
				\(\adjCost \ge 0;\)
			\end{quote}
		\item \label{aff:non-triviality} non-triviality condition
			\begin{quote}
				the adjoint trajectory \(\bigl(\adjDyn\bigr)_{t=0}^{\horizon-1}\) and the pair \(\bigl(\adjCost, \adjFreq\bigr)\) do not simultaneously vanish;
			\end{quote}
		\item \label{aff:optimal dynamics} state and adjoint system dynamics
			\begin{alignat*}{2}
				\optState[t+1] & = \derivative{\genDyn}{\hamiltonian[\adjCost, \adjFreq] \bigl(\adjDyn, t, \optState, \optCon \bigr)} \quad && \text{for } t = 0, \ldots, \horizon-1,\\
				\adjDyn[t-1] & = \derivative{\genState}{\hamiltonian[\adjCost, \adjFreq] \bigl(\adjDyn, t, \optState, \optCon \bigr)} - \adjState \quad && \text{for } t = 1, \ldots, \horizon-1,
			\end{alignat*}
			where \(\adjState \in \dualSpace{\bigl(\R^{\sysDim}\bigr)}\) lies in the dual cone of a tent \(\stateTent(\optState)\) of \(\admStates_{t}\) at \(\optState\);
		\item \label{aff:transversality} transversality conditions
			\begin{align*}
				\derivative{\genState}{\hamiltonian[\adjCost, \adjFreq] \bigl(\adjDyn[0], 0, \optState[0], \optCon[0] \bigr)} - \adjState[0] = 0 \qquad \text{and} \qquad \adjDyn[\horizon-1] = -\adjState[\horizon],
			\end{align*}
			where \(\adjState[0]\) lies in the dual cone of a tent \(\stateTent[0] (\optState[0])\) of \(\admStates_{0}\) at \(\optState[0]\) and \(\adjState[\horizon]\) lies in the dual cone of a tent \(\stateTent[\horizon] (\optState[\horizon])\) of \(\admStates_{\horizon}\) at \(\optState[\horizon]\);

		\item \label{aff:Hamiltonian VI} Hamiltonian maximization condition, pointwise in time,
			\begin{align*}
				\hamiltonian \bigl(\adjDyn, t, \optState, \optCon \bigr) = \max_{\genCon \in \admControl_{t}} \hamiltonian \bigl(\adjDyn, t, \optState, \genCon \bigr) \quad \text{for } t = 0, \ldots, \horizon-1;
			\end{align*}
		\item \label{aff:freq} frequency constraints
			\begin{align*}
				\freqConstr \bigl(\optCon[0], \ldots, \optCon[\horizon-1] \bigr) = 0.
			\end{align*}
	\end{enumerate}
\end{corollary}



\begin{corollary}
\label{co:linear systems}
	Let \(\bigl((\optState)_{t=0}^{\horizon}, (\optCon)_{t=0}^{\horizon-1} \bigr)\) be an optimal state-action trajectory for \eqref{e:aff opt prob} with \(\freqConstr\) as defined in \eqref{e:freq assumptions}. Moreover, suppose that in the optimal control problem \eqref{e:aff opt prob}, the underlying system is linear, state constraints are absent and the end points \(\state_{0}\) and \(\state_{\horizon}\) are fixed; i.e.,
	\begin{equation}
	\label{e:linear sys}
		\sysDyn (\genState, \genCon) = \sys_{t} \genState + \control_{t} \genCon,
	\end{equation}
	and
	\begin{align*}
		& \admStates_{0} \Let \set[\big]{\state_{\rm{ini}}}, \quad \admStates_{\horizon} \Let \set[\big]{\state_{\rm{fin}}},\\
		& \admStates_{t} \Let \R^{\sysDim} \quad \text{for } t = 1, \ldots, \horizon-1.
	\end{align*}
	With the Hamiltonian as defined in \eqref{e:aff ham}, the conditions \ref{aff:non-negativity}, \ref{aff:non-triviality}, \ref{aff:Hamiltonian VI} and \ref{aff:freq} hold, the condition \ref{aff:transversality} is trivially satisfied, and the adjoint dynamics in \ref{aff:optimal dynamics} is given by
	\begin{equation}
		\label{e:adj linear sys}
		\adjDyn[t-1] = \sys_{t} \transp \adjDyn - \adjCost \derivative{\genState}{\cost[t] (\optState, \optCon)} \quad \text{for }t = 1, \ldots, \horizon-1.
	\end{equation}
\end{corollary}



	\section{Linear quadratic optimal control problems}
		\label{sec:LQR problems}
		

In this section we discuss three special cases of linear quadratic (LQ) optimal control problems, all under unconstrained control actions. In \secref{sec:classical LQ} we address the LQ problem with initial and final state constraints and demonstrate that all extremals are normal; this material is standard, but we include it only for the sake of easy reference. \secref{sec:freq LQ} deals with a variation of the LQ state-transfer problem where frequency components of the control sequence are constrained, and we provide conditions for normality of LQ extremals under frequency constraints.

\subsection{Classical LQ problem}
\label{sec:classical LQ}
Consider a linear time-invariant incarnation of \eqref{e:gen sys}:
\begin{equation}
\label{e:LQR sys}
	\state_{t+1} = \sys \state_{t} + \control \conInp_{t}, \quad t = 0, \ldots, \horizon-1,
\end{equation}
where \(\state_{t} \in \R^{\sysDim}\) is the state, \(\conInp_{t} \in \R^{\conDim}\) is the control input at time \(t\), and the system matrix \(\sys \in \R^{\sysDim \times \sysDim}\) and the control matrix \(\control \in \R^{\conDim \times \conDim}\) are known. Consider the following finite horizon LQ problem with unconstrained control actions for the system \eqref{e:LQR sys} given an initial state \(\state_{0} = \initState\):
\begin{equation}
\label{e:LQR prob}
	\begin{aligned}
		& \minimize_{(\conInp)_{t=0}^{\horizon-1}} && \sum_{t=0}^{\horizon-1} \biggl(\half \inprod{\state_{t}}{\stateCost \state_{t}} + \half \inprod{\conInp_{t}}{\controlCost \conInp_{t}} \biggr) \\ 
		& \sbjto &&
		\begin{cases}
			\text{controlled dynamics } \eqref{e:LQR sys},			&\\
			\state_{0} = \initState,
		\end{cases}
	\end{aligned}
\end{equation}
where \(\controlCost \in \R^{\conDim \times \conDim}\) is a given positive definite matrix and \(\stateCost \in \R^{\sysDim \times \sysDim}\) is a given positive semi-definite matrix.

The solution of the LQ problem \eqref{e:LQR prob} can be obtained by using Bellman dynamic programming (DP) principle and algorithm \cite[Chapter 1]{Bertsekas}. This is sketched below in \eqref{e:Bellman Optimality}, and it gives sufficient conditions for optimality of a control sequence \((\optCon)_{t=0}^{\horizon-1}\):
The DP algorithm gives us, with \(\R^{\sysDim} \ni \state \mapsto \BellCost_{t} (\state) \in \R\) denoting the optimal cost-to-go at stage \(t\),
\begin{equation}
\label{e:Bellman Optimality}
\begin{aligned}
	& \BellCost_{\horizon} (\state) \Let 0 \quad \text{for } \state \in \R^{\sysDim} \quad \text{and}\\
	& \BellCost_{t} (\state) = \half \inprod{\state}{\stateCost \state} + \min_{\conInp \in \R^{\conDim}} \biggl( \half \inprod{\conInp}{\controlCost \conInp} + \BellCost_{t+1} (\sys \state + \control \conInp) \biggr) \quad \text{for } t = 0, \ldots, \horizon-1
\end{aligned}
\end{equation}
The fact that the minimum in \eqref{e:Bellman Optimality} is attained follows from the assumption that \(\controlCost\) is positive definite. The following solution of \eqref{e:Bellman Optimality} can be derived readily: for \(t = \horizon-1, \ldots, 0\),
\begin{equation}
\label{e:LQR via DP}
	\begin{cases}
		\BellOpt[\horizon] = 0 \in \R^{\sysDim \times \sysDim},\\
		\BellFeed = - (\controlCost + \control \transp \BellOpt[t+1] \control) \inverse \control \transp \BellOpt[t+1] \sys,\\
		\BellOpt = (\sys + \control \BellFeed) \transp \BellOpt[t+1] (\sys + \control \BellFeed) + \BellFeed \transp \controlCost \BellFeed + \stateCost,\\
		\optCon = \BellFeed \state_{t}.
	\end{cases}
\end{equation}
It is worth noting that the feedback matrix \(\BellFeed\) in \eqref{e:LQR via DP} is independent of any state information, and depends only on how much longer it takes to reach the final stage and the cost-per-stage matrices \(\stateCost\) and \(\controlCost\). 

We employ the classical PMP \cite[Theorem 16]{ref:Bol-75} to \eqref{e:LQR prob}: The Hamiltonian function for \eqref{e:LQR prob} is
\begin{align*}
	& \dualSpace{\R^{\sysDim}} \times \Nz \times \R^{\sysDim} \times \R^{\conDim} \ni (\genDyn, s, \genState, \genCon) \mapsto\\
	& \qquad \hamiltonian[\genCost] (\genDyn, s, \genState, \genCon) \Let \inprod{\genDyn}{\sys \genState + \control \genCon} - \genCost \biggl(\half \inprod{\genState}{\stateCost \genState} + \half \inprod{\genCon}{\controlCost \genCon} \biggr) \quad \text{for } \genCost \in \set[\big]{0, 1}. 
\end{align*}
If \(\bigl((\optState)_{t=0}^{\horizon}, (\optCon)_{t=0}^{\horizon-1}\bigr)\) is an optimal state-action trajectory, then there exist adjoint sequence \(\bigl(\adjDyn \bigr)_{t=0}^{\horizon-1}\) and \(\adjCost \in \set[\big]{0, 1}\), such that \(\adjCost\) and \(\bigl(\adjDyn \bigr)_{t=0}^{\horizon-1}\) are not simultaneously zero, and the necessary conditions of optimality of the trajectory \(\bigl((\optState)_{t=0}^{\horizon}, (\optCon)_{t=0}^{\horizon-1}\bigr)\) given by the PMP can be written as
\begin{enumerate}[label=(\roman*), leftmargin=*, widest=iii]
	\item the adjoint and state dynamics \ref{pmp:optimal dynamics}:
		\begin{equation}
		\label{e:LQ dynamics}
		\begin{aligned}
			& \optState[t+1] = \sys \optState + \control \optCon, \quad \text{for } t = 0, \ldots, \horizon-1,\\
			& \adjDyn[t-1] = \sys \transp \adjDyn - \adjCost \stateCost \optState, \quad \text{for } t = 1, \ldots, \horizon-1;
		\end{aligned}
		\end{equation}
	\item the Hamiltonian maximization condition \ref{pmp:Hamiltonian VI}: At each stage \(t\),
		\begin{equation}
		\label{e:LQ Ham max}
			\derivative{\genCon} \hamiltonian[\adjCost] (\adjDyn, t, \optState, \optCon) = 0 \quad \Rightarrow \quad \adjCost \controlCost \optCon = \control \transp \adjDyn;
		\end{equation}
	\item boundary conditions for the recursive equations are given by the transversality conditions \ref{pmp:transversality}:
		\[
			\optState[0] = \initState \quad \text{and} \quad \adjDyn[\horizon-1] = 0.
		\]
\end{enumerate}
If \(\adjCost = 0\), the adjoint dynamics in \eqref{e:LQ dynamics} reduces to
\[
	\adjDyn[t-1] = \sys \transp \adjDyn \quad \text{for } t = 1, \ldots, \horizon-1.
\]
Since \(\adjDyn[\horizon-1] = 0\), this would imply that \(\adjDyn = 0\) for all \(t = 0, \ldots, \horizon-1\). In other words, \(\adjCost\) and \(\bigl(\adjDyn\bigr)_{t=0}^{\horizon-1}\) would simultaneously vanish, contradicting the non-triviality condition. Hence, there are no abnormal solutions to the PMP in this case. Substituting \(\adjCost = 1\), we get the following set of equations characterising the optimal state-action trajectory.
\begin{equation}
\label{e:LQR via PMP}
\begin{aligned}
	\begin{cases}
		\optState[t+1] = \sys \optState + \control \optCon \quad & \text{for } t = 0, \ldots, \horizon-1, \\
		\adjDyn[t-1] = \sys \transp \adjDyn - \stateCost \optState & \text{for } t = 1, \ldots, \horizon-1, \\
		\optCon = \controlCost \inverse \control \transp \adjDyn & \text{for } t = 0, \ldots, \horizon-1, \\
		\optState[0] = \initState, \quad \adjDyn[\horizon-1] = 0. &
	\end{cases}
\end{aligned}
\end{equation}

Observe that \eqref{e:LQR via PMP} also characterises the optimal control sequence \((\optCon)_{t=0}^{\horizon-1}\) as a linear feedback of the states, which matches with the solution obtained by solving by dynamic programming as exposed in \cite[Chapter 4]{Bertsekas}. 

\label{sec:normal LQ}
For a certain class of LQ optimal control problems in the absence of state and control constraints, all the candidates for optimality are characterised by the PMP with \(\adjCost = 1\), i.e., normal extremals.\footnote{See Remark \ref{rm:extremals}.} One such example is presented next. Recall that a linear time-invariant system \eqref{e:LQR sys} is controllable if \(\rank \pmat{\control & \ldots & \sys^{\sysDim-1} \control} = \sysDim\).

Consider a variation of the LQ problem \eqref{e:LQR prob} where the goal is to reach a specified final state \(\finState \in \R^{\sysDim}\) at time \(\horizon\):
\begin{equation}
\label{e:LQ ST}
	\begin{aligned}
		& \minimize_{(\conInp)_{t=0}^{\horizon-1}} && \sum_{t=0}^{\horizon-1} \biggl(\half \inprod{\state_{t}}{\stateCost \state_{t}} + \half \inprod{\conInp_{t}}{\controlCost \conInp_{t}} \biggr) \\ 
		& \sbjto &&
		\begin{cases}
			\text{controlled dynamics } \eqref{e:LQR sys},			&\\
			\state_{0} = \initState, \quad \state_{\horizon} = \finState.
		\end{cases}
	\end{aligned}
\end{equation}

\begin{proposition}
\label{p:normality}
	If the underlying system \((\sys, \control)\) in \eqref{e:LQ ST} is controllable and \(\horizon \ge \sysDim\), then all the optimal state-action trajectories are normal.
\end{proposition}

\subsection{Normality of LQ state transfer under frequency constraints}
\label{sec:freq LQ}

Let us consider a third variation of the LQ optimal control problem \eqref{e:LQR prob} with constraints on the frequency components of the control sequence but no state and control constraints. We assume that our frequency constraints stipulate that certain frequency components are set to \(0\). We know (cf. \secref{sec:opt prob}, \eqref{e:frequency components}) that there are \(\horizon\) frequency components in a control sequence of length \(\horizon\), and let us select \(\LQFreqDim\) of these to be zero. Recall from \eqref{e:full freq constraints} that such constraints can be written as
\[
	\sum_{t=0}^{\horizon-1} \freqDer \conInp_{t} = 0,
\]
where \(\freqDer\) are defined appropriately corresponding to the \(\LQFreqDim\) frequencies chosen to be eliminated as discussed in \secref{sec:opt prob}.

Consider
\begin{equation}
\label{e:LQ ST freq}
	\begin{aligned}
		& \minimize_{(\conInp)_{t=0}^{\horizon-1}} && \sum_{t=0}^{\horizon-1} \biggl(\half \inprod{\state_{t}}{\stateCost \state_{t}} + \half \inprod{\conInp_{t}}{\controlCost \conInp_{t}} \biggr) \\ 
		& \sbjto &&
		\begin{cases}
			\text{controlled dynamics } \eqref{e:LQR sys},			&\\
			\sum_{t=0}^{\horizon-1} \freqDer \conInp_{t} = 0, \\
			\state_{0} = \initState, \quad \state_{\horizon} = \finState.
		\end{cases}
	\end{aligned}
\end{equation}
Applying the PMP (cf.\ Theorem \ref{th:main pmp}) to get the necessary conditions of optimality of \(\bigl((\optState)_{t=0}^{\horizon}, (\optCon)_{t=0}^{\horizon-1} \bigr)\), we arrive at the following conditions:

There exist \(\adjCost \in \set[\big]{0, 1}\), \(\adjFreq \in \R^{\freqDim}\), a sequence of adjoint variables \(\bigl(\adjDyn\bigr)_{t=0}^{\horizon-1}\), such that \(\adjCost, \adjFreq\), and \(\bigl(\adjDyn\bigr)_{t=0}^{\horizon-1}\) are not simultaneously zero, and
\begin{equation}
\label{e:LQ freq sol}
	\begin{cases}
		\optState[t+1] = \sys \optState + \control \optCon \quad & \text{for } t = 0, \ldots, \horizon-1, \\
		\adjDyn[t-1] = \sys \transp \adjDyn - \adjCost \stateCost \optState & \text{for } t = 1, \ldots, \horizon-1, \\
		\adjCost \controlCost \optCon = \control \transp \adjDyn - \freqDer \transp \adjFreq & \text{for } t = 0, \ldots, \horizon-1, \\
		\sum_{t=0}^{\horizon-1} \freqDer \optCon = 0, &\\
		\optState[0] = \initState, \quad \text{and} \quad \optState[\horizon] = \finState.
	\end{cases}
\end{equation}
The adjoint variables are free at the boundary, i.e., \(\adjDyn[0]\) and \(\adjDyn[\horizon-1]\) are arbitrary.

\begin{proposition}
	\label{p:freq normality}
	If the underlying system (\(\sys, \control\)) in \eqref{e:LQ ST freq} is controllable, \(\horizon \ge \sysDim\), and the number of frequency constraints \(\freqDim\) satisfies \(\freqDim + \sysDim > \conDim \horizon\), then all the optimal state-action trajectories are abnormal. Conversely, all the optimal state-action trajectories are normal when the reachability matrix \(\pmat{\control & \ldots & \sys^{\horizon-1} \control}\) and the frequency constraints matrix \(\mathscr{F} \linTran \inverse\) have independent rows.
\end{proposition}

	\appendix

	\section{Convex Cones and Separability}
		\label{sec:cones}
	This section deals with defining the basic concepts regarding convex sets used later in developing the necessary conditions for optimality.

	\begin{itemize}[label=\(\circ\), leftmargin=*]
		\item Let \(\genDim\) be a positive integer. Recall that a non-empty subset \(K \subset \R^{\genDim}\) is a \embf{cone} if for every \(y \in K\) and \(\alpha \ge 0\) we have \(\alpha y \in K\). In particular, \(0 \in \R^{\genDim}\) belongs to \(K\). A non-empty subset \(C \subset \R^{\genDim}\) is \embf{convex} if for every \(y, y' \in C\) and \(\theta \in [0, 1]\) we have \((1-\theta)y + \theta y' \in C\).
		\item A hyperplane \(\Gamma\) in \(\R^{\genDim}\) is an (\(\genDim -1\))-dimensional affine subset of \(\R^{\genDim}\). It can be viewed as the level set of a nontrivial linear function \(p : \R^{\genDim} \lra \R\). If \(p\) is given by \(p(\genVar) = \inprod{a}{\genVar}\) for some \(a (\neq 0) \in \R^{\genDim}\), then
			\[ \Gamma \Let \set[\big]{\genVar \in \R^{\genDim} \suchthat \inprod{a}{\genVar} = \alpha}.
			\]
		\item We say that a family \(\set{\genTent_{0}, \genTent_{1}, \ldots, \genTent_{s}}\) of convex cones in \(\R^{\genDim}\) is \embf{separable} if there exists a hyperplane \(\Gamma\) and some \(i \in \set{0, \ldots, s}\) such that the cones \(\genTent_{i}\) and \(\bigcap_{j \neq i} \genTent_{j}\) are on two sides of \(\Gamma\); formally, there exists \(c \in \R^{\genDim}\) and \(i \in \set{0, 1, \ldots, s}\) such that \(K_i \subset \set{y \in \R^{\genDim} \suchthat \inprod{c}{y} \le 0}\) and \(\bigcap_{j \neq i} K_j \subset \set{y\in\R^{\genDim}\suchthat \inprod{c}{y} \ge 0}\). \footnote{\label{Guler} More information on separability can be obtained in \cite{Guler}}
		\item Let \(y \in \R^{\genDim}\). A set \(K \subset\R^{\genDim}\) is a \embf{cone with vertex} \(y\) if it is expressible as \(y + K'\) for some cone \(K' \subset \R^{\genDim}\). In particular, any cone is a cone with vertex \(0 \in \R^{\genDim}\).
		\item Let \(\OMG\) be a nonempty set in \(\R^{\genDim}\). By \(\affHull \OMG\) we denote the set of all affine combinations of points in \(\OMG\). That is,
			\[
				\affHull \OMG = \set[\bigg]{\sum_{i=1}^{k} \theta_{i} \genVar_{i} \suchthat \sum_{i=1}^{k} \theta_{i} = 1, \quad \genVar_{i} \in \OMG \quad \text{for } i = 1, \ldots, k, \text{ and } k \in \N}
			\]
			In other words, \(\affHull \OMG\) is also the smallest affine set containing \(\OMG\). The relative interior \(\relInt \OMG\) of \(\OMG\) denotes the interior of \(\OMG\) relative to the affine space \(\affHull \OMG\).
		\item Let \(M\) be a convex set and \(\vertex \in M\). The union of all the rays emanating from \(\vertex\) and passing through points of \(M\) other than \(\vertex\) is a convex cone with vertex at \(\vertex\). The closure of this cone is called the \embf{supporting cone} of \(M\) at \(\vertex\).
		\item Let \(\genTent \subset \R^{\genDim}\) be a convex cone with vertex at \(\vertex\). By \(\dualCone{\genTent}\) we denote its \embf{polar} (or \embf{dual}) cone defined by
			\begin{equation}
			\label{e:dual cone}
				\dualCone{\genTent} \Let \set[\big]{y \in \dualSpace{\bigl(\R^{\genDim} \bigr)} \suchthat \inprod{y}{x - \vertex} \le 0 \quad \text{for all } x \in \genTent}.
			\end{equation}
			It is clear that \(\dualCone{\genTent}\) is a closed convex cone with vertex at \(\vertex\) in view of the fact that it is an intersection of closed half-spaces:
			\[
				\dualCone{\genTent} = \bigcap_{y\in\genTent} \set[\big]{z\in\dualSpace{\bigl(\R^\genDim\bigr)} \suchthat \inprod{z}{y - \vertex} \le 0}.
			\]
			We adopt the contemporary convention of polarity as given in \cite[p.\ 21]{ref:Cla-13}. Our polars are, therefore, negatives of the polars defined in \cite[p.\ 8]{ref:Bol-75}; consequently and in particular, \(\psi_{0}\) in our Theorem \ref{th:necessary condition} is non-negative while \(\psi_{0}\) in \cite[Theorem\ 16]{ref:Bol-75} is non-positive.
	\end{itemize}

	We need a few results from convex analysis, which we quote from various sources below and for the sake of completeness we provide most of their proofs.

	\begin{theorem}[{{\cite[Theorem 4 on p.\ 8]{ref:Bol-75}}}]
	\label{th:dual cone props}
		Let \(\genTent_{1}, \ldots \genTent_{s}\) be closed convex cones in \(\R^{\genDim}\) with vertex at \(\vertex\). Then 
		\[
			\dualCone{\biggl(\bigcap_{i=1}^{s} \genTent_{i}\biggr)} = \overline{\chull \biggl(\bigcup_{i=1}^{s} \dualCone{\genTent_{i}} \biggr)}.
		\]
		Here \(\overline{S}\) denotes the closure of the set \(S\).
	\end{theorem}

	\begin{proof}
		Let \(\genTent \Let \bigcap_{i=1}^{s} \genTent_{i}\). If \(\genVec \in \dualCone{\genTent}\), then for every \(\genVar \in \genTent\) we have
		\begin{equation}
		\label{e:chull dual vec}
			\inprod{\genVec}{\genVar - \vertex} \le 0.
		\end{equation}
		In particular, the relation \eqref{e:chull dual vec} holds for \(\genVar \in \genTent_{i}\) for each \(i = 1, \ldots, s\). This implies that \(\genVec \in \dualCone{\genTent_{i}}\) for \(i = 1, \ldots, s\). Thus,
		\[
			\genVec \in \bigcap_{i=1}^{s} \dualCone{\genTent_{i}} \subset \chull \biggl( \bigcup_{i=1}^{s} \dualCone{\genTent_{i}} \biggr) \subset \overline{\chull \biggl( \bigcup_{i=1}^{s} \dualCone{\genTent_{i}} \biggr)}.
		\]
		This shows that \( \dualCone{\bigl(\bigcap_{i=1}^{s} \genTent_{i}\bigr)} \subset \overline{\chull \bigl(\bigcup_{i=1}^{s} \dualCone{\genTent_{i}}\bigr)}\).

		Now let us prove the converse inclusion. Let \(\genVec \in \chull \bigl( \bigcup_{i=1}^{s} \dualCone{\genTent_{i}} \bigr)\). Then there exist vectors \(\genVec_{1}, \ldots, \genVec_{k} \in \bigcup_{i=1}^{s} \dualCone{\genTent_{i}}\) such that
		\[
			\genVec = \genVec_{1} + \cdots + \genVec_{k}.
		\]
		Since \(\genVec_{i} \in \bigcup_{i=1}^{s} \dualCone{\genTent_{i}}\), for every \(\genVar \in \genTent\) we have \(\inprod{\genVec_{i}}{\genVar - \vertex} \le 0\) for \(i = 1, \ldots, s\). Thus,
		\begin{align*}
			\inprod{\genVec}{\genVar - \vertex} & = \inprod{\genVec_{1}}{\genVar - \vertex} + \cdots + \inprod{\genVec_{s}}{\genVar - \vertex} \le 0.
		\end{align*}
		Therefore, \(\chull \bigl( \bigcup_{i=1}^{s} \dualCone{\genTent_{i}} \bigr) \subset \dualCone{\bigl( \bigcap_{i=1}^{s} \genTent_{i} \bigr)}\). Since the dual cone \(\dualCone{\bigl( \bigcap_{i=1}^{s} \genTent_{i} \bigr)}\) is a closed convex cone, the closure \(\overline{\chull \bigl( \bigcup_{i=1}^{s} \dualCone{\genTent_{i}} \bigr)}\) is also a subset of \(\dualCone{\bigl( \bigcap_{i=1}^{s} \genTent_{i} \bigr)}\).
	\end{proof}

	\begin{theorem}[{{\cite[Theorem 5 on p.\ 8]{ref:Bol-75}}}]
	\label{th:convex set props}
		Let \(\OMG_{1}, \ldots, \OMG_{s}\) be convex sets in \(\R^{\genDim}\) such that \(\bigcap_{i=1}^{s} \relInt \OMG_{i} \neq \emptyset\). Then
		\begin{enumerate}[label={\rm (\roman*)}, leftmargin=*, align=left, widest=iii]
			\item \label{closure} \(\overline{\bigcap_{i=1}^{s} \OMG_{i}} = \bigcap_{i=1}^{s} \overline{\OMG}_{i}\),
			\item \label{aff hull} \(\affHull \bigl(\bigcap_{i=1}^{s} \OMG_{i} \bigr) = \bigcap_{i=1}^{s} \affHull \OMG_{i}\),
			\item \label{rel int} \(\relInt \bigl(\bigcap_{i=1}^{s} \OMG_{i} \bigr) = \bigcap_{i=1}^{s} \relInt \OMG_{i}\).
		\end{enumerate}
	\end{theorem}

	\begin{proof}
		Let \(\OMG \Let \bigcap_{i=1}^{s} \OMG_{i}\).
		\begin{enumerate}[label={\rm (\roman*)}, leftmargin=*, align=left, widest=iii]
		\item If \(\genVar \in \overline{\OMG}\), then there exists a sequence \(\genVar_{k} \in \OMG\) such that \(\genVar_{k} \ra \genVar\). But,
		\[
			\genVar_{k} \in \OMG \Leftrightarrow \genVar_{k} \in \OMG_{i} \quad \text{for } i = 1, \ldots, s.
		\]
		This means that for each \(i = 1, \ldots, s\), there exists a sequence \(\genVar_{k} \in \OMG_{i}\) with \(\genVar_{k} \ra \genVar\), implying that \(\genVar \in \overline{\OMG}_{i}\). This proves the condition \ref{closure}.

		\item 
			If \(\genVar \in \affHull \bigl(\bigcap_{i=1}^{s} \OMG_{i} \bigr)\), then there exist vectors \(\genVar_{1}, \ldots, \genVar_{k} \in \bigcap_{i=1}^{s} \OMG_{i}\) such that
			\[
				\sum_{j=1}^{k} \theta_{j} \genVar_{j} = \genVar \quad \text{with} \quad \sum_{j=1}^{k} \theta_{j} = 1.
			\]
				Since \(\genVar_{j} \in \bigcap_{i=1}^{s} \OMG_{i}\)  if and only if \(\genVar_{j} \in \OMG_{i}\) for each \(i = 1, \ldots, s\), we have \(\genVar \in \affHull \OMG_{i}\) for each \(i = 1, \ldots, s\).

		\item If \(\genVar \in \relInt \bigl(\bigcap_{i=1}^{s} \OMG_{i}\bigr)\), then there exists an \(\epsilon > 0\) such that
			\[
				\ball{\genVar} \cap \affHull \biggl(\bigcap_{i=1}^{s} \OMG_{i} \biggr) \subset \bigcap_{i=1}^{s} \OMG_{i}.
			\]
			In view of condition \ref{aff hull}, we have the following.
			\begin{align*}
				& \ball{\genVar} \cap \biggl(\bigcap_{i=1}^{s} \affHull \OMG_{i} \biggr) \subset \bigcap_{i=1}^{s} \OMG_{i} \\
				\Leftrightarrow \quad & \ball{\genVar} \cap \affHull \OMG_{i} \subset \bigcap_{i=1}^{s} \OMG_{i} \subset \OMG_{i} \quad \text{for each } i = 1, \ldots, s \\
				\Leftrightarrow \quad & \genVar \in \relInt \OMG_{i} \quad \text{for each } i = 1, \ldots, s \\
				\Leftrightarrow \quad & \genVar \in \bigcap_{i=1}^{s} \relInt \OMG_{i}.
			\end{align*}
		\end{enumerate}
	\end{proof}

	\begin{theorem}[{{\cite[Theorem 3 on p.\ 7]{ref:Bol-75}}}]
	\label{th:convex family vectors}
		Let \(\genTent_{1}, \ldots, \genTent_{s}\) be closed convex cones in \(\R^{\genDim}\) with vertex at 0. If the cone \(\genTent = \chull \bigl( \bigcup_{i=1}^{s} \genTent_{s} \bigr)\) is not closed, then there are vectors \(\lambda_{1} \in \genTent_{1}, \ldots, \lambda_{s} \in \genTent_{s}\), not all of them zero, such that \(\lambda_{1} + \cdots + \lambda_{s} = 0\).
	\end{theorem}

	\begin{proof}
		Let \(\genVar \in \overline{\genTent} \setmin \genTent\). Then there exists a sequence of vectors \(\genVar_{k} \in \genTent\) such that \(\genVar_{k} \ra \genVar\). Since \(\genVar_{k} \in \genTent\), we can write
		\[
			\genVar_{k} = \genVar_{k} \kth[1] + \cdots + \genVar_{k} \kth[s] \quad \text{with } \genVar_{k} \kth[i] \in \genTent_{i}
		\]
		Define \(\alpha_{k} \Let \max_{i} \norm{\genVar_{k} \kth[i]}\). We may assume that \(\alpha_{k} > 0\) for all \(k\). It can be seen that \(\alpha_{k} \ra \infty\) since it would mean that \(\genVar \in \genTent\) otherwise. Let \(y_{k} \kth[i] \Let \frac{1}{\alpha_{k}} \genVar_{k} \kth[i]\).

		Without loss of generality, we may assume that the limits \(y \kth[i] = \lim_{k \to \infty} y_{k} \kth[i]\) exist for \(i = 1, \ldots, s\). Since \(\max_{i} \norm{y_{k} \kth[i]} = 1\) for each \(k\), at least one of the vectors \(y \kth[1], \ldots, y \kth[s]\) is not zero. Moreover since \(\genTent_{i}\) is closed, we have \(y \kth[i] \in \genTent_{i}\).

		Since \(\lim_{k \to \infty} \genVar_{k} = \genVar\) and \(\lim_{k \to \infty} \alpha_{k} = \infty\), we have
		\begin{align*}
			y \kth[1] + \cdots + y \kth[s] & = \lim_{k \to \infty} \biggl( y_{k} \kth[1] + \cdots + y_{k} \kth[s] \biggr) \\
			& = \lim_{k \to \infty} \frac{1}{\alpha_{k}} \biggl(\genVar_{k} \kth[1] + \cdots + \genVar_{k} \kth[s] \biggr) \\
			& = \lim_{k \to \infty} \frac{1}{\alpha_{k}} \genVar_{k} = 0.\qedhere
		\end{align*}
	\end{proof}

	\begin{theorem}[{{\cite[Theorem 6 on p.\ 9]{ref:Bol-75}}}]
	\label{th:convex family sep}
		If a family \(\genTent_{1}, \ldots, \genTent_{s}\) of convex cones with a common vertex at \(\vertex\) is not separable, then \(\bigcap_{i=1}^{s} \relInt \genTent_{i} \neq \emptyset\).
	\end{theorem}
	
	\begin{proof}
		Suppose that \(\bigcap_{i=1}^{s} \relInt \genTent_{i} = \emptyset\) and let \(m < s\) be a positive number such that
		\[
			\bigcap_{i=1}^{m} \relInt \genTent_{i} \neq \emptyset \quad \text{and} \quad \bigcap_{i=1}^{m+1} \relInt \genTent_{i} = \emptyset.
		\]
		By Theorem \ref{th:convex set props} (condition \ref{rel int}), \(\bigcap_{i=1}^{m} \relInt \genTent_{i} = \relInt \bigl(\bigcap_{i=1}^{m} \genTent_{i}\bigr)\). This implies that
		\[
			\bigcap_{i=1}^{m+1} \relInt \genTent_{i} = \biggl(\bigcap_{i=1}^{m} \relInt \genTent_{i} \biggr) \cap \relInt \genTent_{m+1} = \relInt \biggl(\bigcap_{i=1}^{m} \genTent_{i}\biggr) \cap \relInt \genTent_{m+1} = \emptyset.
		\]
		Therefore the convex cones \(\genTent_{m+1}\) and \(\bigcap_{i=1}^{m} \genTent_{i}\) have non-empty interior and hence are separable. This implies that the convex cones \(\genTent_{m+1}\) and \(\bigcap_{i \neq m} \genTent_{i}\) are also separable, which contradicts the assumption that the family of cones \(\set[\big]{\genTent_{1}, \ldots, \genTent_{s}}\) is not separable.
	\end{proof}

	\begin{theorem}[{{\cite[Theorem 2 on p.\ 6]{ref:Bol-75}}}]
	\label{th:dual cone vectors}
		Let \(s \in \N\) and \(\set[\big]{\genTent_{0}, \genTent_{1}, \ldots, \genTent_{s}}\) be a family of convex cones in \(\R^{\genDim}\) with a common vertex \(\vertex\). This family is separable if and only if there exist \(\lambda_{i} \in \dualCone{\genTent_{i}}\) for each \(i = 0, \ldots, s\), not all zero, that satisfy the condition
		\begin{equation}
		\label{e:dc vec req}
			\lambda_{0} + \cdots + \lambda_{s} = 0.
		\end{equation}
	\end{theorem}

	\begin{proof}
		Let \(m\) \((\le s)\) be the least number such that the family of cones \(\set[\big]{\genTent_{0}, \ldots, \genTent_{m}}\) is separable. Renumbering the cones if necessary, let us assume that the cones \(\genTent_{0}\) and \(\genTent_{1} \cap \ldots \cap \genTent_{m}\) are separable. This implies that there exists a hyperplane \(\Gamma\) characterised by a non-zero vector \(\genVec\) such that the cones lie in half-spaces given by
		\[
			H = \set{\genVar \suchthat \inprod{\genVec}{\genVar - \vertex} \le 0}, \quad H' =  \set{\genVar \suchthat \inprod{\genVec}{\genVar - \vertex} \ge 0}.
		\]
		If \(m = 1\), then \(\genTent_{0} \subset H\) and \(\genTent_{1} \subset H'\). This implies \(\genVec \in \dualCone{\genTent_{0}}\) and \(-\genVec \in \dualCone{\genTent_{1}}\). Thus, choosing the vectors as
		\[
			\lambda_{0} = \genVec, \quad \lambda_{1} = -\genVec, \quad \lambda_{2} = \cdots = \lambda_{s} = 0,
		\]
		the required condition \eqref{e:dc vec req} is satisfied. If \(m > 1\), the family of cones \(\set[\big]{\genTent_{1}, \ldots, \genTent_{m}}\) is not separable. By Theorem \ref{th:convex family sep} we have  \(\bigcap_{i=1}^{m} \relInt \genTent_{i} \neq \emptyset\). By Theorem \ref{th:convex set props}-\ref{closure},
		\[
			\overline{\bigcap_{i=1}^{m} \genTent_{i}} = \bigcap_{i=1}^{m} \overline{\genTent}_{i}.
		\]
		Since \(\bigcap_{i=1}^{m} \genTent_{i}\) lies in the closed half-space \(H'\), its closure \(\overline{\bigcap_{i=1}^{m} \genTent_{i}} \subset H'\). Therefore, \(\bigcap_{i=1}^{m} \overline{\genTent}_{i} \subset H'\), which implies that \(- \genVec \in \dualCone{\bigl(\bigcap_{i=1}^{m} \overline{\genTent}_{i}\bigr)}\). By Theorem \ref{th:dual cone props},
		\[
			- \genVec \in \overline{\chull \biggl(\bigcup_{i=1}^{m} \dualCone{\overline{\genTent}_{i}}\biggr)} = \overline{\chull \biggl(\bigcup_{i=1}^{m} \dualCone{\genTent_{i}}\biggr)}.
		\]
		If, on the one hand, \(\chull \bigl(\bigcup_{i=1}^{m} \dualCone{\genTent_{i}}\bigr)\) is closed, then \(- \genVec \in \chull \bigl(\bigcup_{i=1}^{m} \dualCone{\genTent_{i}}\bigr)\), implying that there exist vectors \(\lambda_{i} \in \dualCone{\genTent_{i}}\) for \(i = 1, \ldots, m\) such that
		\[
			- \genVec = \lambda_{1} + \cdots + \lambda_{m}.
		\]
		Choosing \(\lambda_{0} = \genVec\; (\neq 0) \in \dualCone{\genTent_{0}}\) and \(\lambda_{j} = 0\) for \(j = m+1, \ldots, s\), the required condition \eqref{e:dc vec req} is satisfied. If, on the other hand, \(\chull \bigl(\bigcup_{i=1}^{m} \dualCone{\genTent_{i}}\bigr)\) is not closed, then by Theorem \ref{th:convex family vectors}, there exist vectors \(\lambda_{i} \in \dualCone{\genTent_{i}}\) for \(i = 1, \ldots, m\), not all zero, such that
		\[
			\lambda_{1} + \cdots + \lambda_{m} = 0.
		\]
		Selecting \(\lambda_{0} = \lambda_{m+1} = \cdots = \lambda_{s} = 0\) we verify that the condition \eqref{e:dc vec req} is satisfied.

		Conversely, assume that there exist \(\lambda_{i} \in \dualCone{\genTent_{i}}\) for \(i = 0, \ldots, s\) satisfying \eqref{e:dc vec req} and not all of them equal to zero (say \(\lambda_{0} \neq 0\)). Since \(\lambda_{0} \in \dualCone{\genTent_{0}}\), we have \(\inprod{\lambda_{0}}{\genVar - \vertex} \le 0\) for \(\genVar \in \genTent_{0}\). This means that \(\genTent_{0}\) is contained in the half-space \(H = \set[\big]{\genVar \suchthat \inprod{\lambda_{0}}{\genVar - \vertex} \le 0}\). By \eqref{e:dc vec req},
		\begin{align*}
			& \lambda_{0} = - \lambda_{1} - \cdots - \lambda_{s}\\
			\Rightarrow \quad & \inprod{\lambda_{0}}{\genVar - \vertex} = - \inprod{\lambda_{1}}{\genVar - \vertex} - \cdots - \inprod{\lambda_{s}}{\genVar - \vertex}.
		\end{align*}
		For \(\genVar \in \bigcap_{i=1}^{s} \genTent_{i}\), we have \(\genVar \in \genTent_{i}\) for \(i = 1, \ldots, s\). Since \(\lambda_{i} \in \dualCone{\genTent_{i}}\), \(\inprod{\lambda_{i}}{\genVar - \vertex} \le 0\) for each \(i = 1, \ldots, s\). Hence, \(\inprod{\lambda_{0}}{\genVar - \vertex} \ge 0\). This implies that the intersection \(\bigcap_{i=1}^{s} \genTent_{i}\) lies in the half-space \(H' = \set[\big]{\genVar \suchthat \inprod{\lambda_{0}}{\genVar - \vertex} \ge 0}\). In other words, the family of cones \(\set[\big]{\genTent_{0}, \ldots, \genTent_{s}}\) is separable.
	\end{proof}

	\begin{theorem}[{{\cite[Theorem 7 on p.\ 10]{ref:Bol-75}}}]
	\label{th:inseparability}
		Let \(s \in \N\), and for each \(i = 1, \ldots, s\) let \(\subSpace{i} \subset \R^{\genDim}\) be a subspace satisfying \(\subSpace{1} + \cdots + \subSpace{s} = \R^{\genDim}\). For each \(i = 1, \ldots, s\) let \(\subSpace{i}^{\Delta}\) denote the direct sum of all subspaces \(\subSpace{1}, \ldots, \subSpace{s}\)  except \(\subSpace{i}\), and \(\genTent_{i}\) be a convex cone in \(\subSpace{i}\) with a common vertex \(\vertex \in \R^{\genDim}\). If \(N_{i} \Let \chull \bigl( \genTent_{i} \cup \subSpace{i}^{\Delta} \bigr)\) for each \(i\), then \(N_{i}\) is a convex cone, and the family \(\set{N_{i} \suchthat i = 1, \ldots s}\) is inseparable in \(\R^{\genDim}\).
	\end{theorem}

	\begin{proof}
		Suppose that the family \(\set{N_{i} \suchthat i = 1, \ldots, s}\) is separable and let (after renumbering if necessary) \(N_{1}\) is separated in \(\R^{\genDim}\) from the intersection \(\Pi \Let \bigcap_{i=2}^{s} N_{i}\) by the hyperplane \(\Gamma\) characterised by \(\set[\big]{\genVar \suchthat \inprod{a}{\genVar - \vertex} = 0}\). That is,
		\[
			N_{1} \subset H = \set[\big]{\genVar \suchthat \inprod{a}{\genVar - \vertex} \le 0}, \quad \Pi \subset H' = \set[\big]{\genVar \suchthat \inprod{a}{\genVar - \vertex} \ge 0}
		\]
		Since \(\subSpace{1} \subset \subSpace{j}^{\Delta} \subset N_{j}\) for all \(j \neq 1\), \(\subSpace{1} \subset \Pi \subset H'\). Since \(\genTent_{1} \subset N_{1} \subset H\), we have \(\subSpace{1} \subset \Gamma\). For \(j = 2, \ldots, s\), \(\subSpace{j} \subset \subSpace{1}^{\Delta} \subset \N_{1} \subset H\) and \(\genTent_{j} \subset \Pi \subset H'\) and this implies that \(\subSpace{j} \subset \Gamma\). We see that \(\subSpace{i} \subset \Gamma\) for all \(i = 1, \ldots, s\). This leads to an obvious contradiction as the span of subspaces contained in a hyperplane \(\Gamma\) is required to be the full space \(\R^{\genDim}\). Hence, the family \(\set[\big]{N_{i} \suchthat i = 1, \ldots, s}\) is not separable in \(\R^{\genDim}\).
	\end{proof}

	\section{Facts about Tents}
		\label{sec:tents}
			In this section an outline of the method of tents is provided.
	\begin{definition}
		Let \(\OMG\) be a subset of \(\R^{\genDim}\) and let \(\vertex \in \OMG\). A convex cone \(Q \subset \R^{\genDim}\) with vertex \(\vertex\) is a \embf{tent} of \(\OMG\) at \(\vertex\) if there exists a smooth map \(\tentMap\) defined in a neighbourhood of \(\vertex\) such that:\footnote{The theory also works for \(\tentMap\) continuous.}
	\begin{enumerate}[leftmargin=*, align=left]
		\item \(\tentMap (\genVar) = \genVar + o  (\genVar - \vertex )\),\footnote{\label{fn:o-Notation} Recall the Landau notation \(\genFun (\genVar) = o (\genVar)\) that stands for a function \(\genFun(0) = 0\) and \(\lim_{\genVar \to 0} \frac{\abs{\genFun(\genVar)}}{\abs{\genVar}} = 0\).} and
		\item there exists \(\epsilon > 0\) such that \(\tentMap (\genVar) \in \OMG\) for \(\genVar \in Q \cap \ball{\vertex}\).
	\end{enumerate}
	\end{definition}

	We say that a convex cone \(\genTent \subset \R^{\genDim}\) with vertex at \(\vertex\) is a local tent of \(\OMG\) at \(\vertex\) if, for every \(\genVar \in \relInt \genTent\), there is a convex cone \(Q \subset \genTent\) with vertex at \(\vertex\) such that \(Q\) is a tent of \(\OMG\) at \(\vertex\), \(\genVar \in \relInt Q\), and \(\affHull Q = \affHull \genTent\). Observe that if \(\genTent\) is a tent of \(\OMG\) at \(\vertex\), then \(\genTent\) is a local tent of \(\OMG\) at \(\vertex\).

	We need the following theorems on tents in the formulation of our PMP in the sequel.

	\begin{theorem}[{{\cite[Theorem 8 on p.\ 11]{ref:Bol-75}}}]
	\label{th:tangent plane}
		Let \(\OMG\) be a smooth manifold in \(\R^{\genDim}\) and \(\genTent\) the tangent plane to \(\OMG\) at \(\vertex \in \OMG\). Then \(\genTent\) is a tent of \(\OMG\) at \(\vertex\).
	\end{theorem}

	\begin{theorem}[{{\cite[Theorem 9 on p.\ 12]{ref:Bol-75}}}]
	\label{th:half-space tent}
		Given a smooth function \(\genFun : \R^{\genDim} \lra \R\), let \(\vertex\) be such that \(\derivative{\genVar}{\genFun (\vertex )} \neq 0\). Define sets \(\OMG, \OMG_{0} \in \R^{\genDim}\) as
		\[
			\OMG \Let \set[\big]{ \genVar \in \R^{\genDim} \suchthat \genFun (\genVar ) \le \genFun (\vertex )}, \quad \OMG_{0} \Let \set[\big]{\vertex} \cup \set[\big]{\genVar \in \R^{\genDim} \suchthat \genFun (\genVar ) < \genFun (\vertex )}.
		\]
		Then the half-space \(\genTent\) given by the inequality \(\inprod{\derivative{\genVar}{\genFun (\vertex )}}{\genVar - \vertex} \le 0\) is a tent of both \(\OMG\) and \(\OMG_{0}\) at \(\vertex\).
	\end{theorem}

	\begin{theorem}[{{\cite[Theorem 10 on p.\ 12]{ref:Bol-75}}}]
	\label{th:supporting cone tent}
		Let \(\OMG \in \R^{\genDim}\) be a convex set and let \(\genTent\) be its supporting cone at \(\vertex \in \OMG\). Then \(\genTent\) is a local tent of \(\OMG\) at \(\vertex\).
	\end{theorem}
	
	\begin{proof}
		Let \(\genVar \in \relInt \genTent\), \(\genVar \neq \vertex\). By definition of supporting cone, there exists \(\genVar' \in \OMG\) such that \(\genVar\) lies on the ray emanating from \(\vertex\) and passing through \(\genVar'\). Since \(\genVar \in \relInt \genTent\), we also have that \(\genVar' \in \relInt \OMG\). Consider a small ball \(\ball[\delta]{\genVar'}\) choosing around \(\genVar'\) choosing \(\delta\) such that \(\vertex \not \in \ball[\delta]{\genVar'}\) and \(\ball[\delta]{\genVar'} \cap \affHull \genTent \subset \OMG\). Consider a cone \(Q\) consisting of rays emanating from \(\vertex\) and passing through points in \(\ball[\delta]{\genVar'} \cap \genTent\). Since \(\OMG\) is a convex set and the points in \(\ball[\delta]{\genVar'} \cap \genTent\) lie in \(\OMG\), there exists an \(\epsilon > 0\) such that \(\ball{\vertex} \cap Q \subset \OMG\). It can be seen that \(Q\) is a tent of \(\OMG\) at \(\vertex\) (the tent map can be considered to be the identity map). It is clear that \(\genVar \in \relInt Q\) and \(\affHull Q = \affHull \genTent\). Therefore, for every \(\genVar \in \relInt \genTent\), there is a tent \(Q\) of \(\OMG\) with vertex at \(\vertex\) containing \(\genVar\) in its interior and satisfying \(\affHull Q = \affHull \genTent\), indicating that \(\genTent\) is a local tent of \(\OMG\) at \(\vertex\).
	\end{proof}

	\begin{theorem}[{{\cite[Theorem 12 on p.\ 14]{ref:Bol-75}}}]
	\label{th:separability}
		Let \(\OMG_{0}, \OMG_{1}, \ldots, \OMG_{s}\) be  subsets of \(\R^{\genDim}\) with a common point \(\vertex\), and \(\genTent_{0}, \genTent_{1}, \ldots, \genTent_{s}\) local tents of these sets at \(\vertex\). If the family of cones \(\set[\big]{\genTent_{0}, \genTent_{1}, \ldots, \genTent_{s}}\) is inseparable and at least one of the cones is not a plane, then there exists \(\genVar' \in \OMG_{0} \cap \OMG_{1} \cap \ldots \cap \OMG_{s}\) and \(\genVar' \neq \vertex\).
	\end{theorem}

	\begin{proposition}
	\label{p:optimality condition}
		A function \(\genFun (\genVar )\) considered on the set \(\Sigma = \OMG_{1} \cap \ldots \cap \OMG_{s}\), attains its minimum at \(\vertex\) if and only if
		\[
			\Omega_{0} \cap \OMG_{1} \cap \ldots \cap \OMG_{s} = \set[\big]{\vertex},
		\]
		where \(\Omega_{0} \Let \set[\big]{\vertex} \cup \set[\big]{\genVar \in \R^{\genDim} \suchthat \genFun (\genVar) < \genFun (\vertex)}\).
	\end{proposition}

	\begin{proof}
		Suppose that there exists a point \(\genVar' \in \OMG_{0} \cap \Sigma\), \(\genVar' \neq \vertex\). Since \(\genVar' \in \OMG_{0}\), \(\genFun (\genVar') < \genFun (\vertex)\). But since \(\genVar' \in \Sigma\), \(\vertex\) is not a minimum point of \(\genFun (\genVar)\) on \(\Sigma\). If \(\vertex\) is not a minimum of \(\genFun (\genVar)\) on \(\Sigma\), then there exists a point \(\genVar' \in \Sigma\) satisfying \(\genFun (\genVar') < \genFun (\vertex)\). This implies that \(\genVar' \in \OMG_{0}\) and the intersection \(\OMG_{0} \cap \Sigma \neq \set[\big]{\vertex}\).
	\end{proof}

	\begin{theorem}[{{\cite[Theorem 16 on p.\ 20]{ref:Bol-75}}}]
	\label{th:necessary condition}
	Let \(\OMG_{1}, \ldots, \OMG_{s}\) be subsets of \(\R^{\genDim}\) and let \(\bigcap_{k=1}^s \OMG_k\ni\genVar\mapsto \genFun (\genVar )\in\R\) be a smooth function. Let \(\Sigma = \bigcap_{k=1}^s\OMG_{k}\), let \(\vertex \in \Sigma\), and let \(\genTent_{i}\) be a local tent of \(\OMG_{i}\) at \(\vertex\) for \(i = 1, \ldots, s\). If \(\genFun\) attains its minimum relative to \(\Sigma\) at \(\vertex\), then there exist vectors \(\lambda_{i} \in \dualCone{\genTent_{i}}\) for \(i = 1, \ldots, s\) and \(\psi_{0} \in \R\) satisfying
		\[
			\psi_{0} \derivative{\genVar}{\genFun (\vertex )} + \lambda_{1} + \cdots + \lambda_{s} = 0
		\]
		such that \(\psi_{0} \ge 0\), and if \(\psi_{0} = 0\), then at least one of the vectors \(\lambda_{1}, \ldots, \lambda_{s}\) is not zero.
	\end{theorem}
	\begin{proof}
		If \(\derivative{\genVar}{\genFun (\vertex)} = 0\), choosing \(\psi_{0} = 1\) and \(\lambda_{1} = \cdots = \lambda_{s} = 0\) will satisfy the given. We assume, therefore, that \(\derivative{\genVar}{\genFun (\vertex )} \neq 0\). Consider the set \(\Omega_{0} \Let \set[\big]{\vertex} \cup \set[\big]{\genVar \in \R^{\genDim} \suchthat \genFun (\genVar ) < \genFun (\vertex )} \), and let \(\genTent_{0}\) be the half-space in \(\R^{\genDim}\) defined by the inequality \(\inprod{\derivative{\genVar}{\genFun (\vertex )}}{\genVar - \vertex} \le 0\). By Theorem \ref{th:half-space tent} the set \(\genTent_{0}\) is a tent of \(\OMG_{0}\) at \(\vertex\). Since at least one of the tents \(\genTent_{0}, \ldots, \genTent_{s}\) is not a plane by assumption, Proposition \ref{p:optimality condition} asserts that if \(\vertex\) is a minimum of \(\genFun\) relative to \(\Sigma\), then \(\Omega_{0} \cap \ldots \cap \OMG_{s}\) is the singleton set \(\set[\big]{\vertex}\). By Theorem \ref{th:separability} the tents \(\genTent_{0}, \ldots, \genTent_{s}\) are separable, since otherwise the intersection \(\Omega_{0} \cap \ldots \cap \OMG_{s}\) would consist a point \(\genVar' \neq \vertex\). Theorem \ref{th:dual cone vectors} now asserts that there exist vectors \(\lambda_{0} \in \dualCone{\genTent_{0}}, \ldots, \lambda_{s} \in \dualCone{\genTent_{s}}\), not all zero, such that
		\[
			\lambda_{0} + \cdots + \lambda_{s} = 0.
		\]
		The condition follows by noting that \(\lambda_{0} = \psi_{0} \derivative{\genVar}{\genFun (\vertex )}\) and \(\psi_{0} \ge 0\).
	\end{proof}

	\section{Proof of Main Result}
		\label{sec:proofs}

		\subsection{Version 1}
			\label{sec:version 1}
				We convert the optimal control problem \eqref{e:opt prob} into a relative extremum problem in a suitable higher-dimensional space. To that end, we define a generic variable
	\begin{equation}
	\label{e:dummy process}
		\dummy \Let ( \dummyState[0], \ldots, \dummyState[\horizon], \dummyCon[0], \ldots, \dummyCon[\horizon-1] ) \in \overbrace{\R^{\sysDim} \times \ldots \times \R^{\sysDim}}^{\horizon+1 \text{ factors}} \times \underbrace{\R^{\conDim} \times \ldots \times \R^{\conDim}}_{\horizon \text{ factors}},
	\end{equation}
	and let \(\proDim \Let \sysDim (\horizon + 1) + \conDim \horizon\) for the rest of this section. We further compress the vector on the right hand side of \eqref{e:dummy process} by writing \(\dummy \Let \bigl(\fullDummyState, \fullDummyCon \bigr)\) for \(\fullDummyState \Let (\dummyState[0], \ldots, \dummyState[\horizon])\) and \(\fullDummyCon \Let (\dummyCon[0], \ldots, \dummyCon[\horizon-1])\). First, we define the standard projection maps from \(\dummy \in \R^{\proDim}\) to the individual factors \(\dummyState \in \R^{\sysDim}\) and \(\dummyCon \in \R^{\conDim}\) in the following way:
	\begin{equation}
		\label{e:marginals from joint}
		\begin{cases}
			\proj{\state}(\dummy) \Let \dummyState \quad \text{for } t = 0, \ldots, \horizon,\\
			\proj{\conInp}(\dummy) \Let \dummyCon \quad \text{for } t = 0, \ldots, \horizon-1.
		\end{cases}
	\end{equation}
	In terms of the notations in \eqref{e:dummy process} and \eqref{e:marginals from joint}, we \emph{lift} the objective function in \eqref{e:opt prob} to a performance index of the joint variables
	\begin{equation}
	\label{e:process cost}
	\begin{aligned}
		& \R^{\proDim} \ni ( \dummyState[0], \ldots, \dummyState[\horizon], \dummyCon[0], \ldots, \dummyCon[\horizon-1] ) \teL \process \mapsto\\
		& \qquad \proCost ( \process ) \Let \sum_{t=0}^{\horizon} \cost( \dummyState, \dummyCon) = \sum_{t=0}^{\horizon-1} \cost\bigl( \proj{\state} (\process ), \proj{\conInp} (\process )\bigr).
	\end{aligned}
	\end{equation}

	Second, we define constraint sets \(\proState, \proCon \subset \R^{\proDim}\) such that if \(\dummy = \bigl(\fullDummyState, \fullDummyCon \bigr) \in \proState \cap \proCon\) in the notation of \eqref{e:dummy process}, then the \(t\)-th factor \(\dummyState[t]\) of \(\fullDummyState\) is constrained to the set \(\admStates_{t}\) and the \(t\)-th factor \(\dummyCon[t]\) of \(\fullDummyCon\) is constrained to the set \(\admControl_{t}\); to wit,
	\begin{equation}
		\label{e:proState and proCon def}
	\begin{aligned}
		& \proState \Let \set[\big]{\dummy \in \R^{\proDim} \suchthat \proj{\state} (\dummy ) \in \admStates_{t}} \quad \text{for } t = 0, \ldots, \horizon,\\
		& \proCon \Let \set[\big]{\dummy \in \R^{\proDim} \suchthat \proj{\conInp} (\dummy ) \in \admControl_{t}} \quad \text{for } t = 0, \ldots, \horizon-1.
	\end{aligned}
	\end{equation}
	Observe that for \(\dummy \in \proState[s]\), the coordinates \(\proj{\state} (\dummy )\) for \(t \neq s\) and all the \(\proj[\tau]{\conInp} (\dummy )\) are arbitrary. Similarly, for \(\dummy \in \proCon[s]\), all the coordinates \(\proj{\state} (\dummy )\) and \(\proj[\tau]{\conInp} (\dummy )\) for \(\tau \neq s\) are arbitrary. We say that \(\proState\) and \(\proCon\) are \emph{lifts} of \(\admStates_{t}\) and \(\admControl_{t}\), respectively.

	Third, we define maps \(\dynProj : \R^{\proDim} \lra \R^{\sysDim}\) for \(t = 0, \ldots, \horizon-1\), to \emph{lift} the dynamics of the system \eqref{e:gen sys} to \(\R^{\proDim}\) in the following way:
	\begin{equation}
	\label{e:process dynamics}
	\begin{aligned}
		& \overbrace{\R^{\sysDim} \times \ldots \times \R^{\sysDim}}^{\horizon+1 \text{ factors}} \times \underbrace{\R^{\conDim} \times \ldots \times \R^{\conDim}}_{\horizon \text{ factors}} \ni (\dummyState[0], \ldots \dummyState[\horizon], \dummyCon[0], \ldots \dummyCon[\horizon-1]) \teL \dummy \mapsto\\
		& \qquad \dynProj (\dummy ) \Let \sysDyn (\dummyState, \dummyCon ) - \dummyState[t+1] = \sysDyn \bigl( \proj{\state}(\dummy), \proj{\conInp} (\dummy) \bigr) - \proj[t+1]{\state}(\dummy) \in \R^{\sysDim}.
	\end{aligned}
	\end{equation}
	By definition, therefore, a vector \(\dummy = (\dummyState[0], \ldots, \dummyState[\horizon], \dummyCon[0], \ldots, \dummyCon[\horizon-1])\in\R^{\proDim}\) satisfies \(\dynProj[t] (\dummy ) = 0\) for all \(t = 0, \ldots, \horizon-1\), if and only if \(\dummyState[t+1] = \sysDyn[t](\dummyState[t], \dummyCon[t])\) for all \(t = 0, \ldots, \horizon-1\). We define a family of sets
	\begin{equation}
		\label{e:proDyn def}
		\proDyn \Let \set[\big]{\dummy \in \R^{\proDim} \suchthat \dynProj (\dummy ) = 0} \quad \text{for } t = 0, \ldots, \horizon-1.
	\end{equation}

	Finally, we define the \emph{lift} of the frequency constraints on the control trajectories:
	\begin{equation}
	\label{e:freq proj}
	\begin{aligned}
		& \overbrace{\R^{\sysDim} \times \ldots \times \R^{\sysDim}}^{\horizon+1 \text{ factors}} \times \underbrace{\R^{\conDim} \times \ldots \times \R^{\conDim}}_{\horizon \text{ factors}} \ni (\dummyState[0], \ldots \dummyState[\horizon], \dummyCon[0], \ldots \dummyCon[\horizon-1]) \teL \dummy \mapsto\\
		& \qquad \freqProj (\dummy) \Let \freqConstr \bigl( \proj[0]{\conInp} (\dummy), \ldots, \proj[\horizon-1]{\conInp} (\dummy) \bigr) = \freqConstr (\dummyCon[0], \ldots, \dummyCon[\horizon-1] ).
	\end{aligned}
	\end{equation}
	We define a \embf{process} \(\process\) to be the concatenation of a control trajectory \((\conInp_{0}, \ldots, \conInp_{\horizon} )\) and its corresponding state trajectory \((\state_{0}, \ldots, \state_{\horizon})\) traced by the system according to \eqref{e:gen sys} as
	\[
		\process \Let (\state_{0}, \ldots, \state_{\horizon}, \conInp_{0}, \ldots, \conInp_{\horizon}).
	\]
	A process \(\process\) satisfying the frequency constraints belongs to the set \(\proFreq\) defined by
	\begin{equation}
		\label{e:proFreq def}
		\proFreq \Let \set[\big]{\dummy \in \R^{\proDim} \suchthat \freqProj (\dummy) = 0}.
	\end{equation}

	Employing the lifts and the notations introduced in \eqref{e:process cost}, \eqref{e:proState and proCon def}, \eqref{e:proDyn def}, and \eqref{e:proFreq def}, we state the optimal control problem \eqref{e:opt prob} equivalently as the following relative extremum problem:
	\begin{equation}
	\label{e:sys rel extr}
	\begin{aligned}
		& \minimize_{\process\in\R^{\genDim}} && \proCost  (\process )\\
		& \sbjto &&	
		\begin{cases}
			\process \in \Sigma,\\
			\Sigma \Let \bigl(\bigcap_{t=0}^{\horizon-1} \proDyn \bigr) \cap \bigl(\bigcap_{t=0}^{\horizon}\proState \bigr) \cap \bigl(\bigcap_{t=0}^{\horizon-1} \proCon \bigr) \cap \proFreq.
		\end{cases}
	\end{aligned}
	\end{equation}
	In the sequel \(\optProc\) will denote a solution of the relative extremum problem \eqref{e:sys rel extr}, comprising of the optimal control trajectory \( (\optCon )_{t=0}^{\horizon-1}\) that solves \eqref{e:opt prob} and the resulting optimal state trajectory \( (\optState )_{t=0}^{\horizon}\).
	
	Define 
	\[
		\subLevel{\optProc} \Let \set{\optProc} \cup \set[\big]{\process \in \R^{\proDim} \suchthat \proCost (\process ) < \proCost (\optProc)}.
	\]
	By Proposition \ref{p:optimality condition}, \(\optProc\) solves \eqref{e:sys rel extr} if and only if \(\Sigma \cap \subLevel{\optProc} = \set{\optProc}\). Let 
	\[
		\fullDynTent, \fullStateTent, \fullConTent, \text{ and }\fullFreqTent \text{ be tents of the sets }\proDyn, \proState, \proCon,\text{ and }\proFreq \text{ at }\optProc, \text{ respectively}.
	\]
	(The sets \(\fullDynTent, \fullStateTent, \fullConTent\), and \(\fullFreqTent\) depend on \(\optProc\), of course, but for notational simplicity we do not explicitly depict the dependence of these sets on \(\optProc\) in what follows.) By Theorem \ref{th:half-space tent} the half-space given by
	\begin{equation}
	\label{e:subLevelTent}
		\subLevelTent \Let \set[\bigg]{\process \in \R^{\proDim} \suchthat \inprod{\derivative{\process}{\proCost (\optProc )}}{\process - \optProc} \le 0}
	\end{equation}
	is a tent of \(\subLevel{\optProc}\) at \(\optProc\).

	\begin{proposition}
	\label{p:sys sep}
		The family of tents \(\set[\big]{\subLevelTent, \fullFreqTent} \cup \bigl\{\fullDynTent\bigr\}_{t=0}^{\horizon-1} \cup \bigl\{\fullStateTent\bigr\}_{t=0}^{\horizon} \cup \bigl\{\fullConTent\bigr\} _{t=0}^{\horizon-1}\) is separable.
	\end{proposition}
	\begin{proof}
		The assertion follows from Proposition \ref{p:optimality condition} and Theorem \ref{th:separability}. Indeed, since the tent \(\subLevelTent\) is a half-space, (and therefore, not a plane,) the family \(\set[\big]{\subLevelTent, \fullFreqTent} \cup \bigl\{\fullDynTent \bigr\}_{t=0}^{\horizon-1} \cup \bigl\{\fullStateTent \bigr\}_{t=0}^{\horizon}  \cup \bigl\{\fullConTent \bigr\} _{t=0}^{\horizon-1}\) of tents satisfies the hypothesis of Theorem \ref{th:separability}. If the family is not separable, then the intersection \(\Sigma \cap \subLevel{\optProc}\) contains a point \(\process'\) different from \(\optProc\). This means that \(\Sigma \cap \subLevel{\optProc} \neq \set{\optProc}\). But then, this contradicts optimality of \(\optProc\) (cf. Proposition \ref{p:optimality condition}).
	\end{proof}

	\begin{proposition}
	\label{p:sys dual cone}
		There exist vectors
		\begin{itemize}[label=\(\circ\), leftmargin=*]
			\item \(\dualCost \in \dualCone{\bigl(\subLevelTent \bigr)}\),
			\item \(\dualDyn \in \dualCone{\bigl(\fullDynTent \bigr)}\) for \(t=0, \ldots, \horizon-1\),
			\item \(\dualState \in \dualCone{\bigl(\fullStateTent \bigr)}\) for \(t=0, \ldots, \horizon\),
			\item \(\dualCon \in \dualCone{\bigl(\fullConTent \bigr)}\) for \(t=0, \ldots, \horizon-1\), and
			\item \(\dualFreq \in \dualCone{\bigl(\fullFreqTent\bigr)}\),
		\end{itemize}
		not all zero, such that
		\begin{equation}
		\label{e:sys dual cone}
			\dualCost + \sum_{t=0}^{\horizon-1} \dualDyn + \sum_{t=0}^{\horizon} \dualState + \sum_{t=0}^{\horizon-1} \dualCon + \dualFreq = 0.
		\end{equation}
	\end{proposition}
	\begin{proof}
		Since the family of cones \(\set[\big]{\subLevelTent, \fullFreqTent} \cup \bigl\{\fullDynTent\bigr\}_{t=0}^{\horizon-1} \cup \bigl\{\fullStateTent \bigr\}_{t=0}^{\horizon} \cup \bigl\{\fullConTent \bigr\}_{t=0}^{\horizon-1}\) is separable in view of Proposition \ref{p:sys sep}, by Theorem \ref{th:dual cone vectors} there exist vectors in the dual cones of each of the cones in \(\set[\big]{\subLevelTent, \fullFreqTent} \cup \bigl\{\fullDynTent \bigr\}_{t=0}^{\horizon-1} \cup \bigl\{\fullStateTent \bigr\}_{t=0}^{\horizon} \cup \bigl\{\fullConTent \bigr\}_{t=0}^{\horizon-1}\) that satisfy \eqref{e:sys dual cone}.
	\end{proof}

	We observe that since \(\process \in \subLevelTent\) satisfies the inequality
	\[
		\inprod{\derivative{\process}{\proCost (\optProc )}}{\process - \optProc} \le 0,
	\]
	in view of \eqref{e:subLevelTent}, every vector in the dual cone \(\dualCone{\bigl(\subLevelTent \bigr)}\) is of the form
	\[
		\dualCost = \adjCost \derivative{\process}{\proCost (\optProc )},
	\]
	where \(\adjCost \ge 0\).

	\begin{proposition}
	\label{p:springers}
		If \(\optProc\) is a solution of the relative extremum problem \eqref{e:sys rel extr}, there exist \(\adjCost \ge 0\) and dual vectors 
		\begin{itemize}[label=\(\circ\), leftmargin=*]
			\item \(\dualDyn \in \dualCone{\bigl(\fullDynTent \bigr)}\) for \(t=0,\ldots, \horizon-1\),
			\item \(\dualState \in \dualCone{\bigl(\fullStateTent \bigr)}\) for \(t=0, \ldots, \horizon\),
			\item \(\dualCon \in \dualCone{\bigl(\fullConTent \bigr)}\) for \(t=0, \ldots, \horizon-1\), and
			\item \(\dualFreq \in \dualCone{\bigl(\fullFreqTent \bigr)}\),
		\end{itemize}
		such that
		\begin{equation}
		\label{e:springers equation}
			\adjCost \derivative{\process}{\proCost (\optProc )} + \sum_{t=0}^{\horizon-1} \dualDyn + \sum_{t=0}^{\horizon} \dualState + \sum_{t=0}^{\horizon} \dualCon + \dualFreq = 0.
		\end{equation}
		In particular, if \(\adjCost = 0\), then at least one of the vectors \(\bigl\{\dualDyn\bigr\} _{t=0}^{\horizon-1}, \bigl\{\dualState\bigr\} _{t=0}^{\horizon}, \bigl\{\dualCon\bigr\} _{t=0}^{\horizon-1}, \dualFreq\) is not zero.
	\end{proposition}
	\begin{proof}
		Follows at once from the arguments in the proof of Theorem \ref{th:necessary condition}.
	\end{proof}

	\begin{proposition}
	\label{p:sys insep}
		The family of tents \(\bigl\{\fullStateTent\bigr\}_{t=0}^{\horizon} \cup \bigl\{\fullConTent \bigr\}_{t=0}^{\horizon-1}\) is not separable.
	\end{proposition}
	\begin{proof}
		Define the subspaces \(\subSpace{\state_{s}}\), \(s = 0,\ldots, T\), and \(\subSpace{\conInp_{s}}\), \(s = 0, \ldots, T-1\), as:
		\begin{align*}
			& \subSpace{\state_{s}} \Let \left\{ \process \in \R^{\proDim} \left|\,
				\begin{aligned}
					& \proj{\state} (\process) = 0 \quad \text{for } t \in \set{0, \ldots, \horizon} \setmin \set{s} ,\quad \text{and }\\
					& \proj{\conInp} (\process ) = 0 \quad \text{for } t = 0, \ldots, \horizon-1
				\end{aligned}
			\right. \right\},\\
			& \subSpace{\conInp_{s}} \Let \left\{ \process \in \R^{\proDim} \left|\,
				\begin{aligned}
					& \proj{\state} (\process ) = 0\text{ for }t = 0, \ldots, T,\text{ and }\\
					& \proj{\conInp} (\process ) = 0\text{ for }t \in \set{0, \ldots, T-1} \setmin \set{s}
				\end{aligned}
			\right. \right\}.
		\end{align*}
		Observe that \(\subSpace{\state_{0}} + \cdots + \subSpace{\state_{\horizon}} + \subSpace{\conInp_{0}} + \cdots + \subSpace{\conInp_{\horizon-1}} = \R^{\proDim}\). Consider the subspaces \(\subSpace{\state_{t}}^{\Delta}\) and \(\subSpace{\conInp_{t}}^{\Delta}\) of \(\R^{\proDim}\) defined by:
		\begin{align*}
			& \subSpace{\state_{t}}^{\Delta} \Let \set[\big]{\process \in \R^{\proDim} \suchthat \proj{\state} (\process ) = 0} \quad \text{for } t = 0, \ldots, \horizon,\\
			& \subSpace{\conInp_{t}}^{\Delta} \Let \set[\big]{\process \in \R^{\proDim} \suchthat \proj{\conInp} (\process ) = 0} \quad \text{for } t = 0, \ldots, \horizon-1.
		\end{align*}
		\[
			\subSpace{\state_{t}}^\Delta = \overset{T+1 \text{ factors}}{\overbrace{\R^{\sysDim} \times \ldots \times \R^{\sysDim} \times \underset{(t+1) \text{-th factor}}{\underbrace{\{0\}}} \times \R^{\sysDim} \ldots \times \R^{\sysDim}}} \times\overset{T \text{ factors}}{\overbrace{\R^{\conDim} \times \ldots \times \ldots \times \R^{\conDim}}}.
		\]
		\[
			\subSpace{\conInp_{t}}^\Delta = \overset{T+1 \text{ factors}}{\overbrace{\R^{\sysDim} \times \ldots \times \R^{\sysDim}}} \times\overset{T \text{ factors}}{\overbrace{\R^{\conDim} \times \ldots \times \R^{\conDim} \times \underset{(t+1) \text{-th factor}}{\underbrace{\{0\}}} \times \R^{\conDim} \times \ldots \times \R^{\conDim}}}.
		\]

		Let \(\stateTent  (\optState )\) be a local tent of \(\admStates_{t}\) at \(\optState\) and let \(\conTent  (\optCon )\) be a local tent of \(\admControl_{t}\) at \(\optCon\). Observe that the inclusions \(\stateTent \subset \proj{\state}\bigl(\subSpace{\state_{t}}\bigr)\) for \(t = 0, \ldots, \horizon\), and \(\conTent\subset \proj{\conInp}\bigl(\subSpace{\conInp_{t}}\bigr)\) for \(t = 0, \ldots, \horizon-1\), hold.
		
		We now construct a family of tents \(\fullStateTent\) and \(\fullConTent\) in the following way:
		\begin{equation}
		\label{e:constraint tents}
		\begin{aligned}
			& \fullStateTent \Let \set[\big]{\process \in \R^{\proDim} \suchthat \proj{\state} (\process ) \in \stateTent ( \proj{\state} (\optProc ) )} \quad \text{for } t = 0, \ldots, \horizon,\\
			& \fullConTent \Let \set[\big]{\process \in \R^{\proDim} \suchthat \proj{\conInp} (\process ) \in \conTent ( \proj{\conInp} (\optProc ) )} \quad \text{for } t = 0, \ldots, \horizon-1.
		\end{aligned}
		\end{equation}
		Let us lift the tents \(\stateTent (\optState)\) as follows:
		\[
			\widetilde{\stateTent} (\optState) \Let \Bigl( \overset{T+1 \text{ factors}}{\overbrace{\{0\} \times \ldots \times \{0\} \times \underset{(t+1)\text{-th factor}}{\underbrace{\stateTent(\optProc)}} \times \{0\} \times \ldots \times \{0\}}} \times \overset{T \text{ factors}}{\overbrace{\{0\} \times \ldots \times \{0\} \times \ldots \times \{0\}}} \Bigr)
		\]
		Observe that \(\widetilde{\stateTent} (\optState) \in \subSpace{\state_{t}}\).
		Therefore,
		\[
			\fullStateTent = \widetilde{\stateTent} (\optState) \cup \subSpace{\state_{t}}^{\Delta}
		\]
		and similarly for \(\fullConTent\). Since \(\stateTent\) and \(\conTent\) are convex cones, it follows that the tents \(\fullStateTent\) and \(\fullConTent\) satisfy the hypothesis of Theorem \ref{th:inseparability}.

		Theorem \ref{th:inseparability} asserts that the family of tents \(\bigl\{\fullStateTent\bigr\}_{t=0}^{\horizon} \cup \bigl\{\fullConTent\bigr\}_{t=0}^{\horizon-1}\) is inseparable, and this establishes the claim.
	\end{proof}

	\begin{remark}
		Note that any sub-family of an inseparable family of cones is also inseparable. Thus, in addition to the family of tents \(\bigl\{\fullStateTent\bigr\}_{t=0}^{\horizon} \cup \bigl\{ \fullConTent\bigr\}_{t=0}^{\horizon-1}\) being inseparable, we have that the families \(\bigl\{\fullStateTent\bigr\}_{t=0}^{\horizon}\), \(\bigl\{ \fullConTent\bigr\}_{t=0}^{\horizon-1}\) are both individually inseparable.
	\end{remark}

	The following proposition constitutes the keystone of our proof of the main Theorem \ref{th:main pmp}.
	\begin{proposition}
	\label{p:keystone}
		If \(\optProc\) is an optimal process of \eqref{e:sys rel extr}, then there exist \(\adjCost \ge 0\) and dual vectors
		\begin{itemize}[label=\(\circ\), leftmargin=*]
			\item \(\dualDyn \in \dualCone{\bigl(\fullDynTent \bigr)}\) for \(t=0, \ldots, \horizon-1\),
			\item \(\dualState \in \dualCone{\bigl(\fullStateTent \bigr)}\) for \(t=0, \ldots, \horizon\), and
			\item \(\dualFreq \in \dualCone{\bigl(\fullFreqTent \bigr)}\),
		\end{itemize}
		such that 
		\begin{equation}
		\label{e:keystone equation}
			\inprod{- \adjCost \derivative{\process}{\proCost (\optProc )} + \sum_{t=0}^{\horizon-1} \dualDyn - \sum_{t=0}^{\horizon} \dualState - \dualFreq}{\perturb{\process}} \le 0,
		\end{equation}
		for every vector \(\perturb{\process}\) such that \(\optProc + \perturb{\process} \in \bigcap_{t=0}^{\horizon-1} \fullConTent\). In particular, if \(\adjCost = 0\), then at least one of \(\bigl\{ \dualDyn\bigr\} _{t=0}^{\horizon-1}\) and \(\dualFreq\) is not zero.
	\end{proposition}
	\begin{proof}
		From Proposition \ref{p:springers} we infer that there exist \(\adjCost \ge 0\) and dual vectors
		\begin{itemize}[label=\(\circ\), leftmargin=*]
			\item \(\dualDyn \in \dualCone{\bigl(\fullDynTent \bigr)}\) for \(t=0, \ldots, \horizon-1\),
			\item \(\dualState \in \dualCone{\bigl(\fullStateTent \bigr)}\) for \(t=0, \ldots, \horizon\),
			\item \(\dualCon \in \dualCone{\bigl(\fullConTent \bigr)}\) for \(t = 0, \ldots, \horizon-1\), and
			\item \(\dualFreq \in \dualCone{\bigl(\fullFreqTent \bigr)}\),
		\end{itemize}
		satisfying \eqref{e:springers equation}, such that if \(\adjCost = 0\), then  at least one of the vectors 
		\[
			\bigl\{\dualDyn\bigr\}_{t=0}^{\horizon-1}, \bigl\{ \dualState\bigr\}_{t=0}^{\horizon}, \bigl\{\dualCon\bigr\}_{t=0}^{\horizon-1}, \dualFreq
		\]
		is not zero. From Proposition \ref{p:sys insep} we know that the family of cones \(\bigl\{\fullStateTent\bigr\}_{t=0}^{\horizon} \cup \bigl\{\fullConTent\bigr\}_{t=0}^{\horizon-1}\) is inseparable. Observe that if \(\adjCost = 0\) and all of \(\bigl\{ \dualDyn\bigr\}_{t=0}^{\horizon-1}\) and \(\dualFreq\) are zero, then the vectors \(\bigl\{\dualCon\bigr\}_{t=0}^{\horizon}, \bigl\{\dualCon\bigr\}_{t=0}^{\horizon-1}\) in the dual cones of the family \(\bigl\{\fullStateTent\bigr\}_{t=0}^{\horizon} \cup \bigl\{\fullConTent\bigr\}_{t=0}^{\horizon}\), not all zero, satisfy
		\[
			\sum_{t=0}^{\horizon} \dualState + \sum_{t=0}^{\horizon-1} \dualCon = 0.
		\]
		In view of Theorem \ref{th:dual cone vectors}, this contradicts the fact that the family of cones \(\bigl\{\fullStateTent\bigr\}_{t=0}^{\horizon} \cup \bigl\{\fullConTent\bigr\}_{t=0}^{\horizon-1}\) is inseparable. This establishes the final assertion.

		We now establish the main assertion. From \eqref{e:springers equation} we have, for any \(\perturb{\process}\in\R^{\proDim}\),
		\[
			\inprod{\adjCost \derivative{\process}{\proCost (\optProc )} + \sum_{t=0}^{\horizon-1} \dualDyn + \sum_{t=0}^{\horizon} \dualState + \dualFreq}{\perturb{\process}} = - \sum_{t=0}^{\horizon-1} \inprod{\dualCon}{\perturb{\process}}.
		\]
		If \(\perturb{\process}\) is a vector such that \(\optProc + \perturb{\process} \in \bigcap_{t=0}^{\horizon-1} \fullConTent\), then \(\optProc + \perturb{\process} \in \fullConTent\) for each \(t = 0, \ldots, \horizon-1\). Since \(\dualCon \in \dualCone{\bigl(\fullConTent \bigr)}\), by definition we have \(\inprod{\dualCon}{\perturb{\process}} \le 0\) for each \(t = 0, \ldots, \horizon-1\), leading to
		\begin{align*}
			& \inprod{\adjCost \derivative{\process}{\proCost (\optProc )} + \sum_{t=0}^{\horizon-1} \dualDyn + \sum_{t=0}^{\horizon} \dualState + \dualFreq}{\perturb{\process}} = - \sum_{t=0}^{\horizon-1} \inprod{\dualCon}{\perturb{\process}} \ge 0\\
			\Rightarrow \quad & \inprod{- \adjCost \derivative{\process}{\proCost (\optProc)} + \sum_{t=0}^{\horizon-1} \bigl(- \dualDyn \bigr) - \sum_{t=0}^{\horizon} \dualState - \dualFreq}{\perturb{\process}} \le 0
		\end{align*}
		Observe that the dual cones \(\dualCone{\bigl(\fullDynTent \bigr)}\) are subspaces and hence, if \(\dualDyn \in \dualCone{\bigl(\fullDynTent \bigr)}\), then \(\widetilde{\dualDyn} \Let -\dualDyn \in \dualCone{\bigl(\fullDynTent \bigr)}\). And this proves the proposition.
	\end{proof}

	Before we delve into the final result that helps us prove Theorem \ref{th:main pmp}, we make some observations on the dual vectors and the gradient matrices. We have the following characterisation of the dual vectors \(\bigl\{\dualDyn\bigr\}_{t=0}^{\horizon-1}\), \(\bigl\{\dualState\bigr\}_{t=0}^{\horizon}\) and \(\dualFreq\).
	\begin{itemize}[leftmargin=*, align=left]
		\item By the construction in \eqref{e:constraint tents}, for \(\optProc + \perturb{\process} \in \fullStateTent\), the coordinates \(\proj[s]{\state} (\perturb{\process} )\) for \(s \neq t\) are arbitrary and \(\proj[\tau]{\conInp} (\perturb{\process} )\) are arbitrary. The coordinates \(\proj{\state}(\perturb{\process})\) lie in the cone \(\stateTent (\optProc )\). Since a dual vector \(\dualState \in \dualCone{ \bigl(\fullStateTent \bigr)}\) has to satisfy \eqref{e:dual cone}, for all \(\optProc + \perturb{\process} \in \fullStateTent\)
			\[
				\inprod{\dualState}{\perturb{\process}} \le 0
			\]
			But since \(\proj[s]{\state} (\perturb{\process})\) for \(s \neq t\) and \(\proj[\tau]{\conInp}(\perturb{\process})\) for \(\tau = 0, \ldots, \horizon-1\) are arbitrary, it can be seen that the corresponding coordinates in \(\dualState\) are zeroes, that is, \(\proj[s]{\state} (\dualState ) = 0\) for \(s \neq t\) and \(\proj[\tau]{\conInp} (\dualState ) = 0\) for \(\tau = 0, \ldots, \horizon-1\).

			Since \(\proj{\state} (\perturb{\process}) \in \stateTent (\optState )\), the corresponding coordinate of \(\dualState\), which is \(\proj{\state} (\dualState )\) lies in the dual cone \(\dualCone{\bigl(\stateTent (\optState) \bigr)}\) which we denote by \(\adjState\).
			\[
				\dualState = (0, \ldots, \adjState, \ldots, 0, 0, \ldots, 0)
			\]
		\item By Theorem \ref{th:tangent plane}, the tangent plane of \(\proDyn\) at \(\optProc\) is a tent of \(\proDyn\) at \(\optProc\). Considering \(\fullDynTent\) to be the tangent plane, every vector in the corresponding dual cone \(\dualCone{\bigl(\fullDynTent \bigr)}\) is of the form
			\[
				\dualDyn = \biggl(\derivative{\process}{\dynProj (\optProc )} \biggr) \transp \adjDyn, \quad \text{for } t = 0, \ldots, \horizon-1,
			\]
			where \(\adjDyn \in \R^{\sysDim}\).
		\item Similarly, the tangent plane of \(\proFreq\) at \(\optProc\) is a tent of \(\proFreq\) at \(\optProc\). Considering \(\fullFreqTent\) to be the tanget plane, every vector in the dual cone \(\dualCone{\bigl(\fullFreqTent \bigr)}\) is of the form
			\[
				\dualFreq = \biggl(\derivative{\process}{\freqConstr (\optProc)} \biggr) \transp \adjFreq,
			\]
			where \(\adjFreq \in \R^{\freqDim}\).
	\end{itemize}

	From \eqref{e:process cost}, \eqref{e:process dynamics}, \eqref{e:freq proj}, we obtain the components of \(\dualCost, \dualDyn, \dualFreq\) as follows:
	\begin{equation}
	\label{e:derivatives}
	\begin{aligned}
		\begin{dcases}
		\derivative{\dummyState}{\proCost (\optProc)} = \derivative{\dummyState[]}{\cost \bigl(\optState, \optCon \bigr)}, & \derivative{\dummyCon}{\proCost (\optProc)} = \derivative{\dummyCon[]}{\cost \bigl( \optState, \optCon \bigr)},\\
		\derivative{\dummyState}{\dynProj (\optProc )} = \derivative{\dummyState[]}{\sysDyn \bigl(\optState, \optCon \bigr)}, & \derivative{\dummyState}{\dynProj[t+1] (\optProc)} = -I_{\sysDim},\\
		\derivative{\dummyState}{\freqConstr (\optProc)} = 0, & \derivative{\dummyCon}{\freqConstr (\optProc)} = \freqDer, \\
		\derivative{\dummyCon}{\dynProj (\optProc)} = \derivative{\dummyCon}{\sysDyn \bigl(\optState, \optCon \bigr)},
		\end{dcases}
	\end{aligned}
	\end{equation}
	for \(t = 0, \ldots, \horizon-1\), and \(I_{\sysDim}\) being the \(\sysDim \times \sysDim\) identity matrix.

	\begin{proposition}
	\label{p:archway}
		If \(\optProc \Let \bigl((\optState)_{t=0}^{\horizon}, (\optCon)_{t=0}^{\horizon-1} \bigr)\) is an optimal process of the optimal control problem \eqref{e:opt prob}, then there exist \(\adjCost \ge 0\) and dual vectors
		\begin{itemize}[label=\(\circ\), leftmargin=*]
			\item \(\adjDyn \in \dualSpace{\bigl(\R^{\sysDim}\bigl)} \quad\) for \(t = 0, \ldots, \horizon-1\),
			\item \(\adjState \in \dualCone{\bigl(\stateTent (\optState)\bigr)} \quad\) for \(t = 0, \ldots, \horizon\),
			\item \(\adjFreq \in \dualSpace{\bigl(\R^{\freqDim}\bigr)}\),
		\end{itemize}
		such that,
		\begin{enumerate}[label={\rm (\roman*)}, leftmargin=*, align=left, widest=iii]
			\item \label{arch: dynamics}
				\begin{align*}
				- \adjCost \derivative{\genState}{\cost (\optState, \optCon )} + \derivative{\genState}{\sysDyn (\optState, \optCon )} \transp \adjDyn - \adjDyn[t-1] - \adjState = 0 \quad \text{for } t = 1, \ldots, \horizon-1;\\
			\end{align*}
			\item \label{arch:transversality} while \(\adjDyn[0], \adjDyn[\horizon-1]\) satisfy
				\begin{align*}
					& \adjDyn[\horizon-1] = - \adjState[\horizon], \\
					& -\adjCost \derivative{\genState}{\cost[0] (\optState[0], \optCon[0] )} + \derivative{\genState}{\sysDyn[0] (\optState[0], \optCon[0])} \transp \adjDyn[0] - \adjState[0] = 0; \quad \text{and}
				\end{align*}
			\item \label{arch:ham VI}
				\[
					\inprod{- \adjCost \derivative{\genCon}{\cost (\optState, \optCon )} + \derivative{\genCon}{\sysDyn (\optState, \optCon )} \transp \adjDyn - \freqDer \transp \adjFreq }{\perturb{\conInp}_{t}} \le 0
				\]
					for all vectors \(\perturb{\conInp}_{t}\) such that \(\optCon + \perturb{\conInp}_{t} \in \conTent (\optCon)\), for \(t = 0, \ldots, \horizon-1\).
		\end{enumerate}
		In particular, if \(\adjCost = 0\), then at least one of \(\bigl\{\adjDyn\bigr\}_{t=0}^{\horizon-1}\) and \(\adjFreq\) is not zero.
	\end{proposition}
	\begin{proof}
		By Proposition \ref{p:keystone}, there exist vectors \(\bigl\{\adjDyn\bigr\}_{t=0}^{\horizon-1}, \adjFreq\) and \(\adjCost \in \R\), not all zero satistying \eqref{e:keystone equation}. By the construction in \eqref{e:constraint tents}, for \(\optProc + \perturb{\process} \in \fullConTent\), the coordinates \(\proj{\state} (\perturb{\process} )\) are arbitrary. So we choose \(t\)  (\(\in \set[\big]{0, \ldots, \horizon}\)) and a \(\perturb{\process}\) such that \(\proj{\state} (\perturb{\process} ) \in \R^{\sysDim}\) is arbitrary and
		\begin{align*}
			& \proj[s]{\state} (\perturb{\process}) = 0 \quad \text{for } s = 0, \ldots, \horizon, \quad s \neq t,\\
			& \proj[\tau]{\conInp} (\perturb{\process}) = 0 \quad \text{for } \tau = 0, \ldots, \horizon-1.
		\end{align*}
		Let \(\perturb{\state}_{t} \Let \proj{\state} (\perturb{\process})\). When we use this particular collection of \(\perturb{\process}\) in \eqref{e:keystone equation}, only the \(\dummyState\) coordinates in the dual vectors will survive. And the remaining equation is,
		\[
			\inprod{- \adjCost \derivative{\dummyState}{\proCost (\optProc )} + \sum_{s=0}^{\horizon-1} \derivative{\dummyState}{\dynProj (\optProc )} - \sum_{s=0}^{\horizon} \dualState - \derivative{\dummyState}{\freqConstr (\optProc )}}{\perturb{\state}_{t}} \le 0.
		\]
		Using the fact that \(\perturb{\state}_{t}\) can be positive or negative and the results in \eqref{e:derivatives}, we have the following condition for each \(t = 1, \ldots, \horizon-1\):
		\begin{equation}
		\label{e:adjoint equation}
		\begin{aligned}
			& - \adjCost \derivative{\dummyState}{ \sum_{s=0}^{\horizon-1} \cost[s] (\optProc)} + \sum_{s=0}^{\horizon-1} \biggl( \derivative{\dummyState}{\dynProj[s] (\optProc)}\biggr) \transp \adjDyn[s] - \adjState = 0\\
			\Rightarrow \quad & - \adjCost \derivative{\dummyState}{\cost (\optState, \optCon)} + \derivative{\dummyState}{\sysDyn (\optState, \optCon)} - \adjDyn[t-1] - \adjState[t] = 0
		\end{aligned}
		\end{equation}
		Using \(\perturb{\state}_{\horizon-1}\) of the same construction, we get the equation \(\adjDyn[\horizon-1] + \adjState[\horizon] = 0\) and using \(\perturb{\state}_{0}\), we get,
		\[
			- \adjCost \derivative{\dummyState[0]}{\cost[0] (\optState[0], \optCon[0])} + \derivative{\dummyState[0]}{\sysDyn[0] (\optState[0], \optCon[0])} - \adjState[0] = 0
		\]
		This proves the first condition.

		If we take \(\perturb{\process}\) such that its coordinates \(\proj[s]{\state}(\perturb{\process})\) are all zero and \(\proj[\tau]{\conInp}(\perturb{\process})\) are zero for \(\tau = 0, \ldots, \horizon-1\), \(\tau \neq t\). And \(\proj{\conInp}(\perturb{\process}) \teL \perturb{\conInp}_{t}\) is such that \(\optCon + \perturb{\conInp}_{t} \in \conTent (\optCon)\). It is easy to see that the vector \(\perturb{\process}\) thus generated lies in the intersection \(\bigcap_{t=0}^{\horizon-1} \fullConTent\). Thus, using equation \eqref{e:keystone equation} and by the construction, we have
		\begin{equation}
		\label{e:ham max}
			\inprod{- \adjCost \derivative{\dummyCon}{\cost (\optState, \optCon)} + \derivative{\dummyCon}{\sysDyn (\optState, \optCon)} \transp \adjDyn - \freqDer \transp \adjFreq}{\perturb{\conInp}_{t}} \le 0.
		\end{equation}
		This procedure can be repeated with vectors for each \(t = 0, \ldots, \horizon-1\), and the assertion follows.
	\end{proof}

We are finally ready for ready for the proof of our main result.
\begin{proof}[Proof of Theorem \ref{th:main pmp}]
	Observe that from the definition of the Hamiltonian \(\hamiltonian\) in \eqref{e:hamiltonian}, we have
	\begin{equation}
	\label{e:ham der}
	\begin{dcases}
		\derivative{\genState}{\hamiltonian \bigl(\adjDyn, t, \optState, \optCon \bigr)} = - \adjCost \derivative{\genState}{\cost (\optState, \optCon)} + \derivative{\genState}{\sysDyn (\optState, \optCon)} \transp \adjDyn \\
		\qquad\qquad \text{for } t = 0, \ldots, \horizon-1, \quad \text{and}\\
		\derivative{\genCon}{\hamiltonian \bigl(\adjDyn, t, \optState, \optCon \bigr)} = - \adjCost \derivative{\genCon}{\cost (\optState, \optCon)} + \derivative{\genCon}{\sysDyn (\optState, \optCon)} \transp \adjDyn - \freqDer \transp \adjFreq.
	\end{dcases}
	\end{equation}
	The conditions of non-negativity \ref{pmp:non-negativity}, non-triviality \ref{pmp:non-triviality} follow from the statement of the Proposition \ref{p:archway}. From \eqref{e:ham der} and \eqref{e:adjoint equation}, we get the adjoint dynamics in \ref{pmp:optimal dynamics}. The transversality conditions follow from Proposition \ref{p:archway} \ref{arch:transversality}. The equation \eqref{e:ham max} readily provides the Hamiltonian maximisation condition \ref{pmp:Hamiltonian VI}.
\end{proof}

		\subsection{Alternate Proof}
			\label{sec:version 2}
			This section provides an alternate approach, suggested to us by Navin Khaneja, to establish Theorem \ref{th:main pmp}; we include it here for its scientific merit and for completeness.

\begin{proof}[Alternate Proof of Theorem \ref{th:main pmp}]
	Let us define an auxillary system with the dynamics
	\begin{equation}
	\label{e:aux sys}
		\auxState[t+1] = \auxSysDyn (\auxState, \conInp_{t}) \Let \auxState - \freqDer \conInp_{t} \quad \text{for } t = 0, \ldots, \horizon-1,
	\end{equation}
	where \(\auxState \in \R^{\freqDim}\).

	Observe that the frequency constraints in \eqref{e:opt prob}, in view of \eqref{e:freq assumptions}, can now be viewed as the terminal state constraint on the auxillary system \eqref{e:aux sys} as
	\begin{equation}
	\label{e:aux bounds}
		\auxState[0] = 0 \quad \text{and} \quad \auxState[\horizon] = 0.
	\end{equation}


	We can now rewrite the problem \eqref{e:opt prob} into a standard optimal control problem with constraints on control magnitude and states.
	\begin{equation}
	\label{e:aug opt prob}
	\begin{aligned}
		& \minimize_{(\conInp_{t})_{t=0}^{\horizon -1}} && \sum_{t=0}^{\horizon-1} \cost (\state_{t}, \conInp_{t})\\
	& \sbjto &&
	\begin{cases}
		\text{dynamics \eqref{e:gen sys}},\\
		\text{auxillary dynamics \eqref{e:aux sys}},\\
		\state_{t} \in \admStates_{t} \quad \text{for } t = 0, \ldots, \horizon,\\
		\auxState[0] = 0 \quad \text{and} \quad \auxState[\horizon] = 0,\\
		\conInp_{t} \in \admControl_{t} \quad \text{for } t = 0, \ldots, \horizon-1,\\
	\end{cases}
\end{aligned}
	\end{equation}
	For the optimal control problem \eqref{e:aug opt prob}, using the usual PMP formulation, we can define the Hamiltonian as
	\begin{equation}
	\label{e:aug Hamiltonian}
	\begin{aligned}
		& \R \times \dualSpace{\bigl(\R^{\freqDim} \bigr)} \times \dualSpace{\bigl(\R^{\sysDim}\bigr)} \times \Nz \times \R^{\sysDim} \times \R^{\freqDim} \times \R^{\conDim} \ni (\genCost, \genAuxDyn, \genDyn, \genTime, \genAuxState, \genState, \genCon) \mapsto\\
			& \qquad \hamiltonian[\genCost] (\genAuxDyn, \genDyn, \genTime, \genAuxState, \genState, \genCon) \Let \inprod{\genAuxDyn}{\auxSysDyn[\genTime](\genAuxState, \genCon)} + \inprod{\genDyn}{\sysDyn[\genTime](\genState, \genCon)} - \genCost \cost[\genTime](\genState, \genCon) \in \R.
		\end{aligned}
	\end{equation}
	From the assertions of the usual PMP, if \(\bigl((\optAuxState)_{t=0}^{\horizon}, (\optState)_{t=0}^{\horizon}, (\optCon)_{t=0}^{\horizon-1}\bigr)\) is an optimal state-action trajectory of \eqref{e:aug opt prob}, then there exist
	\begin{itemize}[label=\(\circ\), leftmargin=*]
		\item a trajectory \(\bigl(\adjDyn\bigr)_{t=0}^{\horizon-1} \subset \dualSpace{\bigl(\R^{\sysDim}\bigr)}\),
		\item a trajectory \(\bigl(\adjAuxDyn\bigr)_{t=0}^{\horizon-1} \subset \dualSpace{\bigl(\R^{\freqDim}\bigr)}\),
		\item a sequence \(\bigl(\adjState\bigr)_{t=0}^{\horizon} \subset \dualSpace{\bigl(\R^{\sysDim}\bigr)}\) and
		\item \(\adjCost \in \R\)
	\end{itemize}
	satisfying the following conditions:
	\begin{enumerate}[label={\textup{(N-\roman*)}}, leftmargin=*, align=left, widest=iii]
		\item \label{npmp:non-negativity} non-negativity condition
			\begin{quote}
				\(\adjCost \ge 0;\)
			\end{quote}
		\item \label{npmp:non-triviality} non-triviality condition
			\begin{quote}
				the state-adjoint trajectory \(\bigl(\adjDyn\bigr)_{t=0}^{\horizon-1}\), the auxillary state-adjoint trajectory \(\bigl(\adjAuxDyn\bigr)_{t=0}^{\horizon-1}\) and \(\adjCost\) do not simultaneously vanish;
			\end{quote}
		\item \label{npmp:optimal dynamics} state, auxillary state and adjoint system dynamics
				\begin{alignat*}{2}
					\optState[t+1] & = \derivative{\genDyn}{\hamiltonian[\adjCost] \bigl(\adjAuxDyn, \adjDyn, t, \optAuxState, \optState, \optCon \bigr)} \quad && \text{for } t = 0, \ldots, \horizon-1,\\
					\optAuxState[t+1] & = \derivative{\genAuxDyn}{\hamiltonian[\adjCost] \bigl(\adjAuxDyn, \adjDyn, t, \optAuxState, \optState, \optCon \bigr)} \quad && \text{for } t = 0, \ldots, \horizon-1\\
					\adjDyn[t-1] & = \derivative{\genState}{\hamiltonian[\adjCost] \bigl(\adjAuxDyn, \adjDyn, t, \optAuxState, \optState, \optCon \bigr)} - \adjState \quad && \text{for } t = 1, \ldots, \horizon-1,\\
					\adjAuxDyn[t-1] & = \derivative{\genAuxState}{\hamiltonian[\adjCost] \bigl(\adjAuxDyn, \adjDyn, t, \optAuxState, \optState, \optCon \bigr)} \quad && \text{for } t = 1, \ldots, \horizon-1,
				\end{alignat*}
				where \(\adjState\) lies in the dual cone of a tent \(\stateTent( \optState )\) of \(\admStates_{t}\) at \(\optState\); 
			\item \label{npmp:transversality} transversality conditions
				\begin{align*}
					& \derivative{\genState}{\hamiltonian[\adjCost] \bigl(\adjAuxDyn[0], \adjDyn[0], 0, \optAuxState[0], \optState[0], \optCon[0] \bigr)} - \adjState[0] = 0 \qquad \text{and} \qquad \adjDyn[\horizon-1] = -\adjState[\horizon],\\
					& \optAuxState[0] = 0 \qquad \text{and} \qquad \optAuxState[\horizon] = 0;
				\end{align*}
				where \(\adjState[0]\) lies in the dual cone of a tent \(\stateTent[0] (\optState[0])\) of \(\admStates_{0}\) at \(\optState[0]\) and \(\adjState[\horizon]\) lies in the dual cone of a tent \(\stateTent[\horizon] (\optState[\horizon])\) of \(\admStates_{\horizon}\) at \(\optState[\horizon]\);
			\item \label{npmp:Hamiltonian VI} Hamiltonian maximization condition, pointwise in time,
				\begin{align*}
					\inprod{\derivative{\genCon}{\hamiltonian[\adjCost] \bigl(\adjAuxDyn, \adjDyn, t, \optAuxState, \optState, \optCon \bigr)}}{\perturb{\conInp}_{t}} \le 0 \quad \text{whenever } \optCon + \perturb{\conInp}_{t} \in \conTent (\optCon),
				\end{align*}
				where \(\conTent(\optCon)\) is a local tent at \(\optCon\) of the set \(\admControl_{t}\) of admissible actions;
	\end{enumerate}

	Observe that from the definition of Hamiltonian in \eqref{e:aug Hamiltonian} and from \eqref{e:aux sys}, the auxillary state-adjoint dynamics reduces to (for \(t = 1, \ldots, \horizon-1\))
	\begin{align*}
	 	\adjAuxDyn[t-1] & = \derivative{\genAuxState}{\hamiltonian[\adjCost] \bigl(\adjAuxDyn, \adjDyn, t, \optAuxState, \optState, \optCon \bigr)} \\
	 	& = \derivative{\genAuxState}{\biggl(\inprod{\adjAuxDyn}{\auxSysDyn(\optAuxState, \optCon)} + \inprod{\adjDyn}{\sysDyn(\optState, \optCon)} - \adjCost \cost(\optState, \optCon)\biggr)} \\
		& = \derivative{\genAuxState}{\inprod{\adjAuxDyn}{\auxSysDyn (\optAuxState, \optCon)}} = \adjAuxDyn
	\end{align*}
	and \(\adjAuxDyn[\horizon-1]\) can be chosen arbitrarily. This implies the trajectory \(\bigl(\adjAuxDyn\bigr)_{t=0}^{\horizon-1}\) can be replaced by a constant vector, say \(\adjFreq \in \R^{\freqDim}\). That is,
	\begin{equation}
	\label{e:enter the adjFreq}
		\adjAuxDyn[0] = \cdots = \adjAuxDyn[\horizon-1] \teL \adjFreq.
	\end{equation}
	Similarly, using the definition of the Hamiltonian, the condition \ref{npmp:Hamiltonian VI} can be written as
	\begin{equation}
	\label{e:equivalence of Ham VI}
	\begin{aligned}
		& \inprod{\derivative{\genCon}{\hamiltonian[\adjCost] \bigl(\adjAuxDyn, \adjDyn, t, \optAuxState, \optState, \optCon \bigr)}}{\perturb{\conInp}_{t}} \le 0 \\ 
		\Leftrightarrow \quad & \inprod{\derivative{\genCon}{\biggl(\inprod{\adjAuxDyn}{\auxSysDyn(\optAuxState, \optCon)} + \inprod{\adjDyn}{\sysDyn(\optState, \optCon)} - \adjCost \cost(\optState, \optCon)\biggr)}}{\perturb{\conInp}_{t}} \le 0 \\ 
		\Leftrightarrow \quad & \inprod{\derivative{\genCon}{\biggl(\inprod{\adjAuxDyn}{\optAuxState + \freqDer \optCon} + \inprod{\adjDyn}{\sysDyn(\optState, \optCon)} - \adjCost \cost(\optState, \optCon)\biggr)}}{\perturb{\conInp}_{t}} \le 0 \\
		\Leftrightarrow \quad & \inprod{\derivative{\genCon}{\biggl(\inprod{\adjDyn}{\sysDyn (\optState, \optCon)} - \adjCost \cost (\optState, \optCon) + \inprod{\adjFreq}{- \freqDer \optCon}\biggr)}}{\perturb{\conInp}_{t}} \le 0 \\
	\end{aligned}
	\end{equation}
	whenever \(\optCon + \perturb{\conInp}_{t} \in \conTent\).

	Hence, defining a new Hamiltonian as in \eqref{e:hamiltonian}, the conditions \ref{npmp:non-negativity} - \ref{npmp:Hamiltonian VI} transform to the conditions \ref{pmp:non-negativity} - \ref{pmp:freq} as shown below.
	\begin{enumerate}[label=(\roman*), align=left, leftmargin=*, widest=iii]
		\item The non-negativity condition \ref{pmp:non-negativity} is same as the condition \ref{npmp:non-negativity}
		\item Since the non-triviality condition \ref{npmp:non-triviality} asserts that \(\adjCost, \bigl(\adjDyn\bigr)_{t=0}^{\horizon-1}, \bigl(\adjAuxDyn\bigr)_{t=0}^{\horizon-1}\) do not vanish simultaneosly, the non-triviality condition \ref{pmp:non-triviality} follows from \eqref{e:enter the adjFreq}.
		\item It can be observed from the way the Hamiltonian is defined in \eqref{e:hamiltonian}, the optimal state and adjoint dynamics specified in \ref{npmp:optimal dynamics} is same as the one in \ref{pmp:optimal dynamics}.
		\item The transversality conditions in \ref{pmp:transversality} also follow from the definition of Hamiltonian in \eqref{e:hamiltonian} and the conditions \ref{npmp:transversality} on states \((\state)\) and adjoint \((\adjDyn)\).
		\item From \eqref{e:equivalence of Ham VI}, we can see that \ref{npmp:Hamiltonian VI} holds if and only if \ref{pmp:Hamiltonian VI} holds.
		\item The condition \ref{pmp:freq} is another way of writing the transversality conditions on auxillary states and auxillary state-adjoints in \ref{npmp:transversality}. The equivalence follows directly from the dynamics of auxillary states specified by \ref{npmp:optimal dynamics} and the equivalence of the condition \ref{pmp:freq} and the boundary conditions on auxillary states in \ref{npmp:transversality} as shown in \eqref{e:aux bounds}.
	\end{enumerate}
	Our proof is now complete.
\end{proof}

	\section{Proofs of Corollaries}
		\label{sec:corollary proofs}
		\begin{proof}[Proof of Corollary \ref{co:con-affine pmp}]
	The conditions \ref{aff:non-negativity}, \ref{aff:non-triviality}, \ref{aff:optimal dynamics}, \ref{aff:transversality}, and \ref{aff:freq} follow directly from Theorem \ref{th:main pmp}. The Hamiltonian maximization condition, pointwise in time, \ref{aff:Hamiltonian VI} is proved as follows:

	Since \(\admControl_{t}\) is convex, by Theorem \ref{th:supporting cone tent}, the supporting cone \(\supportCone\) of \(\admControl_{t}\) at \(\optCon\) is a local tent of \(\admControl_{t}\) at \(\optCon\). By \ref{pmp:Hamiltonian VI}, for every vector \(\perturb{\conInp_{t}}\) satisfying \(\optCon + \perturb{\conInp_{t}} \in \supportCone\), the optimal actions \(\optCon\), optimal states \(\optState\) and the adjoint vectors \(\adjDyn\) satisfy
\[
	\inprod{\derivative{\genCon}{\hamiltonian \bigl(\adjDyn, t, \optState, \optCon \bigr)}}{\perturb{\conInp}_{t}} \le 0.
\]
	Since the supporting cone \(\supportCone\) includes the set \(\admControl_{t}\), the directions \(\perturb{\conInp_{t}}\) satisfying \(\optCon + \perturb{\conInp_{t}} \in \supportCone\) include all the directions into the set \(\admControl_{t}\) from \(\optCon\). This implies that at \(\optCon\),  the directional derivative \(\derivative{\genCon}{\hamiltonian \bigl(\adjDyn, t, \optState, \optCon \bigr)}\) is non-positive for every direction \(\perturb{\conInp_{t}}\) into the set \(\admControl_{t}\), which is a necessary condition for optimality of \(\hamiltonian \bigl(\adjDyn, t, \optState, \genCon \bigr)\) at \(\optCon\). Note that since \(\affCost (\genState, \cdot)\) is convex, we have
	\[
		\biggl[\frac{\partial^{2}}{\partial \genCon^{i} \partial \genCon^{j}} \hamiltonian \bigl(\adjDyn, t, \optState, \optCon \bigr) \biggr]_{i,j} = - \biggl[\frac{\partial^{2}}{\partial \genCon^{i} \partial \genCon^{j}} \affCost (\optState, \optCon) \biggr]_{i, j} \le 0.
	\]
	Thus, the function \(\hamiltonian \bigl(\adjDyn, t, \optState, \cdot \bigr) : \admControl_{t} \ra \R\) is concave, and hence the necessary condition for optimality is also sufficient.
	The set \(\admControl_{t}\) being compact, the function \(\hamiltonian\) attains its maximum.
\end{proof}

\begin{proof}[Proof of Corollary \ref{co:linear systems}]
	Observe that when \(\admStates_{t} = \R^{\sysDim}\), the dual cone  of \(\stateTent (\optState)\), \(\dualCone{\bigl(\stateTent(\optState) \bigr)} = \set[\big]{0}\) and when \(\admStates_{t}\) is a singleton set, the dual cone \(\dualCone{\bigl(\stateTent(\optState) \bigr)} = \R^{\sysDim}\). Since \(\admStates_{t} = \R^{\sysDim}\) for \(t = 1, \ldots, \horizon-1\), the vectors \(\adjState = 0\) for \(t = 1, \ldots, \horizon-1\). Thus, the adjoint dynamics in \ref{aff:optimal dynamics} specialises to \eqref{e:adj linear sys}.
Since \(\admStates_{0}\) and \(\admStates_{\horizon}\) are singleton sets, the vectors \(\adjState[0]\) and \(\adjState[\horizon]\) are arbitrary and thus the transversality conditions in \ref{aff:transversality} are trivially satisfied.
\end{proof}

	\section{Proofs of LQ Propositions}
		\label{sec:LQ proofs}
		\begin{proof}[Proof of Proposition \ref{p:normality}]
	Since both the initial and the final states are fixed, from the transversality conditions we see that \(\adjDyn[0]\) and \(\adjDyn[\horizon-1]\) can be arbitrary. Suppose that the PMP holds in abnormal form, i.e., \(\adjCost = 0\). In this case the adjoint dynamics equation reduces to the following.
	\begin{align*}
		\adjDyn[t-1] = \sys \transp \adjDyn \quad \text{for } t = 1, \ldots, \horizon-1
	\end{align*}
	The adjoint variable \(\adjDyn\) is given in terms of the \(\adjDyn[\horizon-1]\) as
	\begin{equation}
	\label{e:adj rel}
		\adjDyn = \bigl(\sys \transp \bigr)^{\horizon-t-1} \adjDyn[\horizon-1] \quad \text{for } t = 0, \ldots, \horizon-1.
	\end{equation}
	From the Hamiltonian maximization condition (which is uncosntrained optimization with respect to control variable since there are no control action constraints), we obtain the following conditions.
	\[
		\adjCost \controlCost \optCon = \control \transp \adjDyn \quad \text{for } t = 0, \ldots, \horizon-1.
	\]
	Since \(\adjCost = 0\), by assumption, it follows that
	\begin{align*}
		& \control \transp \adjDyn = 0 \quad \text{for } t = 0, \ldots, \horizon-1, \\
		\Rightarrow \quad & \control \transp \bigl(\sys \transp \bigr)^{\horizon-t-1} \adjDyn[\horizon-1] = 0 \quad \text{for } t = 0, \ldots \horizon-1 \quad \text{in view of \eqref{e:adj rel}.}
	\end{align*}
	This implies that \(\adjDyn[\horizon-1]\) is in the null space of \(\pmat{\control & \sys \control & \ldots & \sys^{\horizon-1} \control} \transp\). But since the pair \((\sys, \control)\) is controllable, the matrix \(\pmat{\control & \sys \control & \ldots & \sys^{\horizon-1} \control}\) has full column rank and thus, its range space (image) is \(\R^{\sysDim}\). Since the range space (image) of a matrix \(C\) is orthogonal to the kernel/null space of its transpose \(C \transp\),  the null space of \(\pmat{\control & \sys \control & \ldots & \sys^{\horizon-1} \control} \transp\) is just the zero vector. This means that \(\adjDyn[\horizon-1] = 0\). From \eqref{e:adj rel}, we see that \(\bigl(\adjDyn\bigr)_{t=0}^{\horizon-1}\) is the zero sequence. But this contradicts the non-triviality assertion of the PMP.
\end{proof}

\begin{proof}[Proof of Proposition \ref{p:freq normality}]
	If \(\adjCost = 0\), then from \eqref{e:LQ freq sol} we have,
	\begin{align*}
		& \control \transp \adjDyn = \freqDer \transp \adjFreq \quad \text{for } t = 0, \ldots, \horizon-1, \quad \text{and}\\
	& \adjDyn[t-1] = \sys \transp \adjDyn \quad \text{for } t = 0, \ldots, \horizon-1.
	\end{align*}
	This means \(\adjDyn = (\sys \transp)^{\horizon-1-t} \adjDyn[\horizon-1]\)  for \(t = 0, \ldots, \horizon-1\) and therefore,
	\[
		\control \transp  (\sys \transp)^{\horizon-1-t} \adjDyn[\horizon-1] = \freqDer \transp \adjFreq \quad \text{for } t = 0, \ldots, \horizon-1.
	\]
	Letting
	\[
		\LQReach \Let \pmat{\control \transp (\sys \transp)^{\horizon-1}\\ \vdots \\ \control \transp } \in \R^{\conDim \horizon \times \sysDim} \quad \text{and} \quad \LQFreq \Let \pmat{\freqDer[0] \transp \\ \vdots \\ \freqDer[\horizon-1] \transp} \in \R^{\conDim \horizon \times \freqDim},
	\]
	we have \(\LQReach \adjDyn[\horizon-1] = \LQFreq \adjFreq\). Note that \(\LQReach\) is the transpose of the reachability matrix and \(\LQFreq = \bigl(\mathscr{F} \linTran \inverse \bigr) \transp\). By assumption, \(\rank \bigl(\LQReach\bigr) = \sysDim\).

	If the equation
	\begin{equation}
	\label{e:abn req}
		\pmat{\LQReach & - \LQFreq} \pmat{\adjDyn[\horizon-1] \\ \adjFreq} = 0
	\end{equation}
	admits a non-trivial solution, then there exist \(\adjCost, \adjFreq\) and \(\bigl(\adjFreq\bigr)_{t=0}^{\horizon-1}\), not all zero, satisfying \eqref{e:LQ freq sol}. Since when \(\adjCost = \) the optimal state-action trajectory \(\bigl((\optState)_{t}^{\horizon}, (\optCon)_{t=0}^{\horizon-1}\bigr)\) is independent of \(\adjFreq\) and \(\bigl(\adjDyn\bigr)_{t=0}^{\horizon-1}\), every feasible solution of \eqref{e:LQ ST freq} is an abnormal solution of PMP.

	The equation \eqref{e:abn req} admits a non-trivial solution only when \(\rank \pmat{\LQReach & - \LQFreq} < \sysDim + \freqDim\). Since \(\rank \pmat{\LQReach & - \LQFreq} = \min \{\sysDim + \freqDim, \conDim \horizon \}\), there exist non-trivial solutions to \eqref{e:abn req} when \(\sysDim + \freqDim > \conDim \horizon\). And when the rows of the reachability matrix and the frequency constraints matrix \(\mathscr{F} \linTran \inverse\) are independent, the \(\rank \pmat{\LQReach & - \LQFreq} = \sysDim + \freqDim\) and there do not exist any non-trivial solutions to \eqref{e:abn req} and thus, all the optimal state-action trajectories are normal.
\end{proof}

	\bibliography{references}
	\bibliographystyle{alpha}

\end{document}